\documentclass[12pt]{article}

\usepackage[a4paper, left=3.5cm, right=3.5cm, top=3.5cm]{geometry}
\usepackage{amsmath}
\usepackage{amsthm} % für proof-Umgebung und Gestaltung von Satz-Umgebungen
\usepackage{amssymb}
\usepackage{todonotes}
\usepackage{setspace}
\usepackage{enumerate}

\usepackage{tikz,pgfplots}
\usetikzlibrary{decorations.pathreplacing}
\usetikzlibrary{calc}
\usepackage{overpic}
\usepackage[skins]{tcolorbox}

\usepackage{authblk}
\usepackage{lipsum}

\usepackage{hyperref}
\usepackage{mathabx}
\newtheorem{theorem}{Theorem}
\newtheorem{definition}[theorem]{Definition}
\newtheorem{corollary}[theorem]{Corollary}
\newtheorem{lemma}[theorem]{Lemma}

\newtheorem{notation}[theorem]{Notation}
\newtheorem{remark}[theorem]{Remark}
\newtheorem{assumption}[theorem]{Assumption}

\DeclareMathOperator{\projvariable}{i}

\DeclareMathOperator{\Xgroundset}{\Upsilon}

\DeclareMathOperator{\RR}{{\mathbb{R}}}
\DeclareMathOperator{\Proj}{\mathcal{P}}

\DeclareMathOperator{\Projind}{\mathfrak{p}}

\DeclareMathOperator{\Projbasic}{P}

\DeclareMathOperator{\sign}{sign}

\newcommand{\dd}[1]{\,\mathrm{d}{#1}}

\newcommand{\supp}[1]{\mathrm{supp}\left(#1\right)}

\newcommand{\rg}[1]{\mathrm{Rg}({#1})}

\newcommand{\myresolution}{1000}

\DeclareMathOperator{\ZZ}{\mathbb{Z}}
\DeclareMathOperator{\NN}{\mathbb{N}}

\DeclareMathOperator{\Radon}{\mathcal{R}}
\DeclareMathOperator{\Fanbeam}{\mathcal{F}}
\DeclareMathOperator{\ExpoFanbeam}{\mathcal{E}}

\DeclareMathOperator{\rotated}{\,\circlearrowleft}

\DeclareMathOperator{\tvariable}{t}
\DeclareMathOperator{\tgroundset}{T}
\DeclareMathOperator{\rayvariable}{r}
\DeclareMathOperator{\raygroundset}{R}
\DeclareMathOperator{\tinverse}{\tau}
\DeclareMathOperator{\rayinverse}{\rho}

\DeclareMathOperator{\stripfkt}{h}

\DeclareMathOperator{\Vone}{V_1}
\DeclareMathOperator{\Vtwo}{V_2}
\DeclareMathOperator{\Vvariable}{{V_{\projvariable}}}

\DeclareMathOperator{\attenuation}{\omega}

\newcommand{\conditionname}{$\mathrm{PPRC}$}

\DeclareMathOperator{\weightfkt}{\zeta}
\DeclareMathOperator{\altweightfkt}{\tilde {\zeta}}

\DeclareMathOperator{\rhoone}{\weightfkt_{1}}
\DeclareMathOperator{\rhotwo}{\weightfkt_{2}}
\DeclareMathOperator{\rhoi}{\weightfkt_{\projvariable}}

\DeclareMathOperator{\altrhoone}{\altweightfkt_{1}}
\DeclareMathOperator{\altrhotwo}{\altweightfkt_{2}}
\DeclareMathOperator{\altrhoi}{\altweightfkt_{\projvariable}}

\DeclareMathOperator{\imgdom}{{\Omega}}
\DeclareMathOperator{\essequ}{\doteq}

\DeclareMathOperator{\gammagroundset}{{\mathcal M}}

\begin{document}

\newcommand\blfootnote[1]{%
\begingroup
\renewcommand\thefootnote{}\footnote{#1}%
\addtocounter{footnote}{-1}%
\endgroup
}

\title{Determination of Range Conditions for General Projection Pair Operators}
\author[1]{Richard Huber}    \author[2]{Rolf Clackdoyle} \author[2]{Laurent Desbat}
\affil[1]{IDea Lab, University of Graz, Austria}

\affil[2]{Univ. Grenoble Alpes, CNRS, Grenoble INP, TIMC }
\date{}
\maketitle
%\address{IOP Publishing, Temple Circus, Temple Way, Bristol BS1 6HG, UK}
%\ead{submissions@iop.org}
%\vspace{10pt}
%\begin{indented}
%\item[]August 2017
%\end{indented}

\addcontentsline{toc}{section}{Abstract}
\begin{abstract}
\blfootnote{richard.huber@uni-graz.at, laurent.desbat@univ-grenoble-alpes.fr and rolf.clackdoyle@univ-grenoble-alpes.fr}
Tomographic techniques are vital in modern medicine, allowing doctors to observe patients’ interior features. Individual steps in the measurement process are modeled by `single projection operators' $\Projind$. These are line integral operators over a collection of curves that covers the regions of interest. 
Then, the entire measurement process can be understood as a finite collection of such single projections, and thus be modeled by an $N$-projections operator $\Projbasic=(\Projind_1,\dots,\Projind_N)$.  The most well-known example of an $N$-projections operator is the restriction of the Radon transform to finitely many projection angles. 
Characterizations of the range of $N$-projections operators are of intrinsic mathematical interest and can also help in practical applications such as geometric calibration, motion detection, or model parameter identification. 
In this work, we investigate the range of projection pair operators $\Proj$ in the plane, i.e., operators formed by two projections ($N=2$) applied to functions in $\RR^2$.
We find that the  set of annihilators to $\rg{\Proj}$ that are regular distributions contains at most one dimension
and a range condition can be explicitly determined by what we refer to as  `kernel conditions'.
With this tool, we examine the exponential fanbeam transform for which no range conditions were known, finding that no (regular) range condition exists, and therefore, arbitrary data can be approximated in an $L^2$ sense by projections of smooth functions.
We also illustrate the use of this theory on a mixed parallel-fanbeam projection pair operator. 
\end{abstract}

\vspace{2pc}
\newpage
%\maketitle

\section{Introduction}
\label{section_introduction}
Projection-based imaging modalities such as Computed Tomography (CT), Positron Emission Tomography (PET) and Single Photon Emission Computed Tomography (SPECT) \cite{ Deans_Radon_applications_1993,BLOKLAND200270,
Emissiontomo_2005, Ammari08_book_methematics_of_medical_imaging, Hsieh_CT_principles} have become indispensable tools in medical investigations.
In computed tomography, a sequence of X-ray images from various positions around the patient is measured with the goal of identifying the overall density distribution inside the body. Roughly speaking, seeing an object only from one direction is insufficient  to identify it fully, but rotating around it and seeing it from all directions might. Doing so requires knowledge of the interaction between the X-rays and the object's density $f$ for each X-ray image. To mathematically model these physical processes, a single projection operator $\Projind$ assigns the integration values of $f$ along straight lines (pointing in the projection direction) to those lines, thereby modeling the attenuation the radiation undergoes while transversing matter for a single projection. Thus, a complete measurement is modeled by an $N$-projections operator $\Projbasic=(\Projind_1,\dots, \Projind_N)$ for $N$ `independent' projections.
Tomographic reconstruction corresponds to the solution of the operator equation $\Projbasic f =g$, where $f$ denotes the density distribution to be determined, while $g$ corresponds to the measured data. 
Various other types of tomography can also be modeled using line integral operators, e.g., the Radon transform, the fanbeam (also divergent beam) transform, or exponential Radon transform; see  \cite[Chapter II]{Natterer:2001:MCT:500773} for an introduction to all three, as well as  \cite{Deans_Radon_applications_1993,V_Aguilar_1995}. 
While many tomographic techniques in three dimensions exist,
% e.g., cone beam tomography \cite{Ning2000FlatPD,JAFFRAY20021337}, 
the investigations in this paper are restricted to the two-dimensional case: tomography of compactly supported functions in the plane, i.e., $f\colon \Omega\to \RR$ is compactly supported in a bounded set $\Omega\subset \RR^2$.

%, with the specifics dependent on choices such as the use of transmission versus emission tomography \cite{Emissiontomo_2005, Ammari08_book_methematics_of_medical_imaging, Natterer:2001:MCT:500773}, or technical details like the design of detectors; the use of collimators, pinhole cameras, etc. \cite{chaves2017introduction,Beekman_pinhole_2007}. 

Knowledge of the range of these projection operators $\Projbasic$  has proven itself useful for various practical applications. (In the applied literature, range conditions are often referred to as `data consistency conditions'.) Notable examples include: 
detection of the motion of small objects \cite{clackdoyle:hal-01967298}; reduction of motion artifacts in CT \cite{1637536,mouchet:hal-03511400};
CT calibration to avoid beam hardening \cite{8364614}; misalignment correction in CT \cite{lesaint:hal-01686646,desbat:hal-03099405, Luo_data_sustained_misalignment,
nguyen:hal-03101279};
estimation of missing projections \cite{Michel_Defrise_2003, new_data_consistency_fanbeam, Ma2017JohnsEC};
identification of attenuation maps for PET \cite{https://doi.org/10.1002/mma.1670150504,737676,6857b0c455474528a1ce35ccfeea3d3b,1208606}; cross-plane PET into lines of response on contiguous 2D slices conversion \cite{M_Defrise_1995,Michel_Defrise_1999}; identification of artifacts in attenuation images for PET \cite{1596627}; estimation of attenuation maps in time of flight \cite{Defrise2012TimeofflightPD};
identification of attenuation factors or attenuation regions in SPECT \cite{1573668924263339008,F_Natterer_1993, Welch_towards_accurate_attenuation_correction_in_spect, C_Mennessier_1999};
correction of depth-dependent collimator effects \cite{301464,293929};
correction of misalignment between SPECT and  CT data \cite{Wells104,https://doi.org/10.1002/mp.15058_2021}.

It is more common to analyze projection operators as having infinitely many projections -- and the associated range conditions might involve all of them -- but real-life measurements only capture finitely many projections. This creates a disconnect between range conditions meant for infinitely many projections and data obtained through finitely many projections. 
Clackdoyle proposed the notion of `projection form range conditions' in \cite{6495510}, meaning any condition only involving finitely many projections, and pointed out their relevance when working with a limited number of projections. 

Thus, beyond its intrinsic mathematical interest, mathematically describing the range $\rg{\Projbasic}$ for $N$-projections operators $\Projbasic$ can help in practical applications. To that end, conditions are needed to verify whether the measured data $g$ satisfy $g\in \rg{\Projbasic}$, i.e., $g=\Projbasic f$ for some suitable $f$. Such conditions are required to be easily verifiable and only rely on knowledge of $g$ and not on $f$ (because they are typically applied prior to reconstruction). 
Since verifying orthogonality between two vectors is simple, the orthogonal complement of the range induces a natural collection of range conditions, which are convenient for testing the consistency of given data $g$.
%Since verifying orthogonality between two vectors is simple, a natural type of range conditions are testing data $g$ against vectors in the range's orthogonal complement.

A classic example of such conditions is the Gelfand-Graev-Helgason-Ludwig moment condition of order zero, stating that in the Radon transform (describing parallel CT), any two projections must contain the same total mass. These conditions were developed by Gelfand and Graev  \cite{WOS:A1960WE64800004} as well as Helgason and Ludwig \cite{10.1007/BF02391776,Ludwig_66}. 
(In the literature, the conditions are commonly referred to as Helgason-Ludwig conditions.)

Similar range conditions are known for other prevalent projection operators, such as the exponential Radon transform and the fanbeam transform \cite{doi:10.1080/01630568308816147,V_Aguilar_1995,new_data_consistency_fanbeam, Levine2010ConsistencyCF,Ludwig_66,helgason2013radon,6495510}.
In contrast, for two projections of the exponential fanbeam transform,  the authors announced in \cite{Huber_fully_3d_2023} that no range conditions exist. Other projection operators have yet to be investigated in that regard.
However, knowledge of range conditions for one projection operator does not inform about range conditions for other projection operators (modeling other types of tomography), so they have to be derived individually.
Annihilators of the range (naively speaking, vectors orthogonal to the range) $V$ can be seen as inducing range conditions by checking $\langle V,g\rangle =0$. Hence, it is natural to ask how big the set of annihilators is, and how to determine annihilators.

Answering these questions for general $N$-projections operators is very challenging. Thus, this paper investigates them for a restricted notion of `Projection Pair Operators' consisting of only two projections; i.e., $N=2$. Moreover, we restrict our investigations to annihilators that are regular distributions, i.e., can be represented as integration against $L^1_\text{loc}$ functions.
Such range conditions for projection pair operators are here referred to as `Projection Pair Range Conditions' (\conditionname s).

%\todo{sparse projection operator to $N$-projections operator. For $N=2$ we call this a projection pair operator}
Given an $N$-projections operator $\Projbasic=(\Projind_1,\dots,\Projind_N)$, one can check the \conditionname s  of all pairs of projections $(\Projind_k,\Projind_l)$ with $k\neq l$ to find inconsistencies; a versatile approach sometimes used in practice.  In this sense, \conditionname s  induce necessary range conditions of $N$-projections operators with more than two projections, although other range conditions might also exist for the $N$-projections operator. 
Thus, the usefulness of \conditionname s transcends tomography with only two projections.

%This work is develops a machinery for identifying \conditionname s, and  recognizing when no such conditions exist. 
%These are based on `kernel conditions' -- equations allowing the identification of orthogonal vectors.
%In view of the large number of applications range conditions found in diverse fields of tomography, it seems evident that such a tool would be highly appreciated.

The paper is organized as follows.
In Section \ref{section_Abstract}, we investigate general
 \conditionname s from a theoretical perspective. In the main theoretical contribution in Theorem \ref{thm_kernel_condition}, we find necessary and sufficient conditions on the form annihilators inducing the range conditions may take, resulting in what we call kernel conditions. Then, Section \ref{section_recover_known_conditions} shows how these results apply to parallel and fanbeam geometries and illustrate the use of kernel conditions for a mixed parallel-fanbeam projection pair. In Section \ref{section_no_condition_for_exponential_fanbeam}, we apply this framework to the exponential fanbeam transform. In Section \ref{Section_exponential_fanbeam_theory}, we show that {\em no \conditionname{} exists for the exponential fanbeam transform}, and any two $L^2$ functions $(g_1,g_2)$ can be approximated arbitrarily well by the exponential fanbeam transforms of smooth functions. In Section \ref{Section_expo_fanbeam_numerics}, we illustrate this fact using numerical simulations.
 %finding that the kernel condition it imposes does not possess a solution.
% Section \ref{section_interval_condition} presents a novel type of range condition for the exponential fanbeam transform. Rather than precisely comparing two projections via linear projection pair functionals, this condition limits the variation such functionals can have. 

\section{Theoretical Investigations of General Projection Pair Operators}
\label{section_Abstract}

In this section, we investigate `projection pair range conditions' (\conditionname s) and their analytical properties. In the preliminary Section \ref{section_projection_pair_preliminaries}, we describe the mathematical setup, introduce single projection operators, and discuss some of their basic properties. In Section \ref{section_projection_pair_identification_kernels}, we move on to pairs of single projection operators and investigate conditions describing an overlap of information between them. To that end, we introduce the concept of projection kernels which induce \conditionname s, and investigate their properties. In particular, Theorem \ref{thm_kernel_condition} characterizes projection kernels via kernel conditions, which provide a general method to determine \conditionname s.

\subsection{Preliminaries and single projection operators}
\label{section_projection_pair_preliminaries}
The general setup can be described as follows. We investigate a function $f$ (e.g., describing the density of an object) in a planar domain $\imgdom$. The concept of a single projection is a collection of measurements of certain integrals over curves (typically straight lines). These curves cover the entire $\imgdom$ and each curve does not self-intersect and does not intersect with any other curve in that projection; i.e., $\imgdom$ is covered bijectively by this set of curves.

\begin{notation}\label{Notation_basic_introduction}
For the purposes of this paper, we use the following notation.
\begin{itemize}
\item[1.a)]  The sets $\imgdom,\gammagroundset\subset \RR^2$ are open, bounded, and connected and will reflect the imaging domain, and a domain of parametrized curves through the imaging domain, respectively.

\item[1.b)] Given such a set $\gammagroundset$, the set $\raygroundset\subset \RR$ denotes the set containing all $\rayvariable\in \RR$ such that there is a $\tvariable\in \RR$ with $(\rayvariable,\tvariable)\in \gammagroundset$.
For fixed $\rayvariable\in \raygroundset$, we set $\tgroundset(\rayvariable):=\{\tvariable \in \RR \ \big | \ (\rayvariable,\tvariable) \in \gammagroundset\}$.  
From the properties of $\gammagroundset$, we note that $\raygroundset$ is an open and bounded interval, while $\tgroundset(\rayvariable)$ is an open bounded non-empty set for each $\rayvariable \in \raygroundset$.
With this notation, we write $\gammagroundset$ as 
$\gammagroundset=\{(\rayvariable,\tvariable)\in \RR^2 \ \big|\ \rayvariable\in \raygroundset \text { and } \tvariable \in \tgroundset(\rayvariable)\}$; see Figure \ref{fig_general_illustration_gamma}.

\item[1.c)] The map $\gamma \colon \gammagroundset \to \imgdom$ describes a $\mathcal{C}^1$-diffeomorphism (i.e., bijective, continuously differentiable map, with continuously differentiable inverse; note that continuously differentiable is equivalent to the function and its partial derivatives being continuous). Moreover, we assume that $\det\left( \frac{\dd \gamma}{\dd{(\rayvariable,\tvariable)}}\right)$ and $\det\left( \frac{\dd \gamma^{-1}}{\dd{x}}\right)$ are bounded on $\gammagroundset$ and $\imgdom$, respectively.
We use the notation $\gamma^{-1}=(\rayinverse,\tinverse)$, i.e., $(\rayinverse(x),\tinverse(x))=\gamma^{-1}(x)$,   which given $x\in \imgdom$ determines the corresponding parametrization $\rayvariable=\rayinverse(x)$ and $\tvariable=\tinverse(x)$ such that $x=\gamma(\rayvariable,\tvariable)$.

\item[2.a)] As usual, $\mathcal{C}^\infty_c(\imgdom)$ refers to infinitely differentiable functions with compact support contained in $\imgdom$, and $L^2(\imgdom)$ denotes the set of all measurable functions $f\colon \imgdom \to \RR$  that satisfy $\|f\|_{L^2}< \infty$. (Recall, $\|f\|_{L^2}^2:={\int_{\imgdom}|f(x)|^2\dd x}$, and strictly speaking, $L^2(\imgdom)$ contains equivalence classes of functions only differing on sets of measure zero; we also say they are equal almost everywhere (abbreviated to `a.e.'), or they differ only on sets of measure zero.)

\item[2.b)] 
The space $L^\infty_c(\raygroundset)$ consists of measurable functions that are essentially bounded (i.e., except on sets of measure zero) and whose support is compactly contained in $\raygroundset$.
Moreover,  $L^1_\text{loc}(\raygroundset)$ denotes the space of measurable functions that are (absolutely) integrable on any compact subset of $\raygroundset$; see, e.g., \cite{evans2010partial}. Note that $L^1_\text{loc}$ is (roughly speaking) the largest space of measurable functions with some integrability requirement, as it has no assumption of compact support and contains all $L^p$ spaces for $p\in [1,\infty]$.

%\item We consider an open bounded connected set $\gammagroundset\subset \RR^2$ that will reflect the parametrizations of curves through the sample.

\end{itemize}

\end{notation}

\begin{definition}
\label{def_projection_curves}
Given $\imgdom$, $\gammagroundset$ and $\gamma$ as in Notation \ref{Notation_basic_introduction}, and a bounded measurable function $\weightfkt\colon \gammagroundset \to [c,C]$ with $0<c<C$, we define  the corresponding single projection operator $\Projind$ such that given a function $f\in \mathcal{C}^\infty_c(\imgdom)$ and $\rayvariable \in \raygroundset$, we have
\begin{equation}\label{equ_def_single_proj_operator}
[\Projind f](\rayvariable) := \int_{\tgroundset (\rayvariable)} f(\gamma(\rayvariable,\tvariable)) \weightfkt(\rayvariable,\tvariable) \dd \tvariable.
\end{equation}

\end{definition}

\begin{figure}
\center 

\begin{tikzpicture}
\newcommand{\globalshift}{1.5}

\newcommand{\eps}{0.7}
\newcommand{\mypi}{3.14}
\newcommand{\A}[2]{(\eps *(cos(deg(\mypi*#1))+1)*(cos(deg(\mypi*#2))+1)/4 + \eps * (cos(deg(\mypi/2*#1))* (cos(deg(\mypi/2*#2)))}

\draw[](0,3) node{\large $\imgdom$};
\draw[](\globalshift*4,3) node{\large $\gammagroundset$};

%\draw[fill,lightgray] (0,0) ellipse (2cm and 2cm);
\fill [gray,rotate=45] (-1.41,-1.41) rectangle (1.41,1.41);

\draw[fill overzoom image=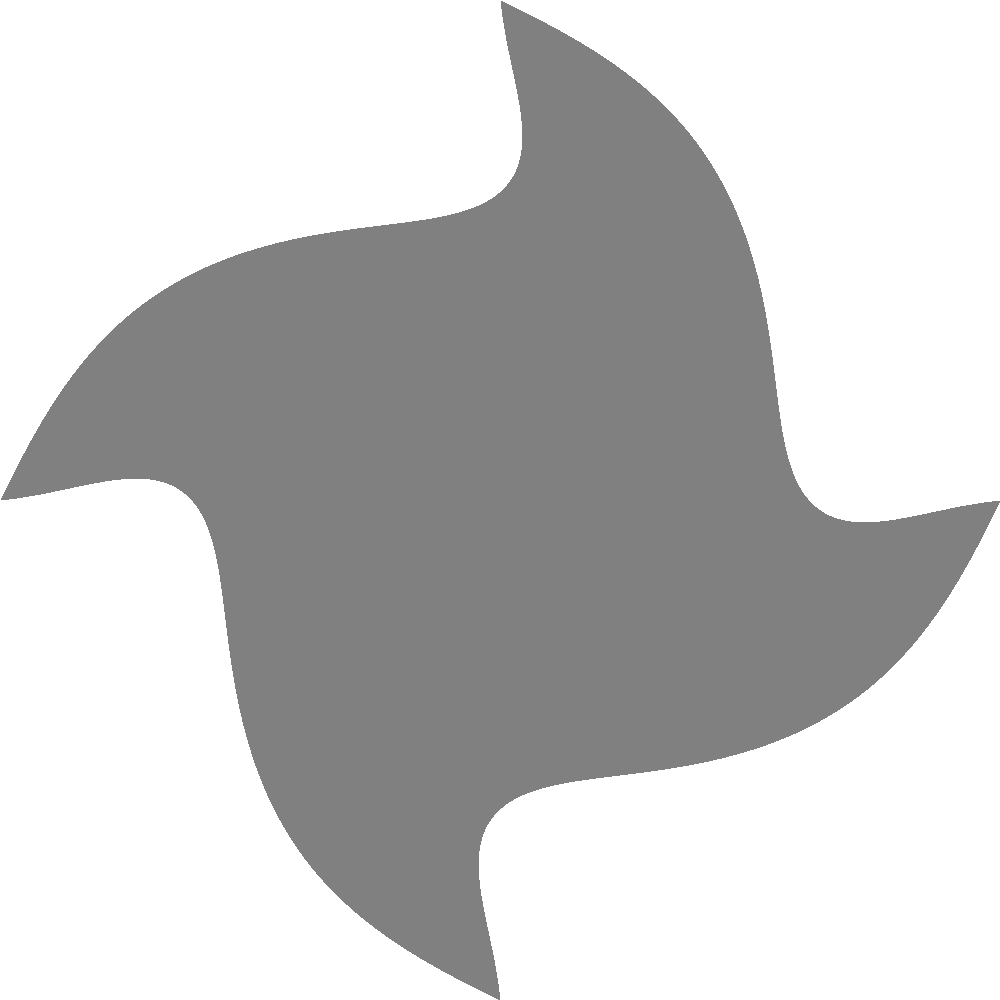] (4*\globalshift-2,-2)rectangle (4*\globalshift+2,2);

\foreach \x/\mycolor in {-1/red,-0.6/orange,-0.2/teal,0.2/cyan,0.6/blue,1/violet}
{
\draw[scale=2, domain=-1:1, ultra thick, smooth, variable=\y, \mycolor] plot ({\x-\y*\A{\x}{\y}}, {\y+\x*\A{\x}{\y}});

\draw[scale=2, domain=-1:1, ultra thick, smooth, variable=\y, \mycolor,xshift=2*\globalshift cm] plot ({\x}, {\y});
}

\newcommand{\centerx}{\globalshift*2+0.6}
\draw [->] (\centerx,-2-0.5) -- (\centerx,2) node [above]{$\tvariable$};

\draw [->] (\centerx,-2-0.5) -- (\globalshift*2-0.5+6,-2-0.5) node [right]{$\rayvariable$};

\draw[thick,orange,decorate,decoration={brace,amplitude=4pt}] (\centerx,487/1000*4-2)-- (\centerx,732/1000*4-2) node[left,midway,yshift=0.00cm,xshift=0.05cm]{$\tgroundset(\rayvariable)$};

\draw[orange,dashed](\centerx,487/1000*4-2) -- (\globalshift*4-0.6*2,487/1000*4-2);

\draw[orange,dashed](\centerx,732/1000*4-2) -- (\globalshift*4-0.6*2,732/1000*4-2);

\draw[thick,magenta,decorate,decoration={brace, amplitude=4pt}] (\globalshift*4+1*2,-2-0.5) -- (\globalshift*4-1*2,-2-0.5) node[below,midway,yshift=0.00cm,xshift=0.05cm]{$\raygroundset$};

\draw[magenta,dashed] (4*\globalshift-2,-2.5) -- (4*\globalshift-2,-2);
\draw[magenta,dashed] (4*\globalshift+2,-2.5) -- (4*\globalshift+2,-2);

\foreach \y/\mycolor in {-1/red,-0.6/orange,-0.2/green,0.2/cyan,0.6/blue,1/violet}
{
\draw[scale=2, domain=-1:1, ultra thick,dashed, smooth, variable=\x] plot ({\x-\y*\A{\x}{\y}}, {\y+\x*\A{\x}{\y}});

\draw[scale=2, domain=-1:1, ultra thick,dashed, smooth, variable=\x,xshift=2*\globalshift cm] plot ({\x}, {\y});
}

\begin{scope}
\draw (3,3.5) node{\large $\gamma$};

\clip (1.8,3) rectangle (5.8,4);
\draw[->,bend right=30, ultra thick](4,3) edge (2,3);
\end{scope}

\end{tikzpicture}

\begin{tikzpicture}[scale=1.2]
\newcommand{\globalshift}{3}

\draw [fill, blue!10] (-1, -1) rectangle (\globalshift+1.4,1.7);

\draw[fill,gray!50] (0,0) circle (1 cm);
\draw (0.8,1.4) node[below,black]{$\imgdom$};
\draw[thick,brown,rotate=45,->](0,0) -- (1.1,0) node[right ]{$\vartheta$};
\draw (0.8,1.4) node[below,black]{$\imgdom$};
\draw[thick,brown,rotate=135,->](0,0) -- (1.1,0) node[above ]{$\vartheta^{\rotated}$};

\begin{scope}
\draw[ ->](0,0) -- (0,1.1) node[above]{$x_2$};
\draw[ ->](0,0) -- (1.1,0) node[right]{$x_1$};
\draw (1.5,1.5) node{\large $\gamma$};

\draw[->,bend right=30, ultra thick](2,1.2) edge (1,1.2);
\end{scope}

\draw[fill,gray!50] (\globalshift,0) circle (1 cm);
\draw (\globalshift-0.7,1.4) node[below,black]{$\gammagroundset$};
\draw[ ->](\globalshift,0) -- (\globalshift,1.1) node[above]{$\tvariable$};
\draw[ ->](\globalshift,0) -- (\globalshift+1.1,0) node[right]{$\rayvariable$};

\foreach \s/\mycolor in {-1/red,-0.6/orange,-0.2/teal,0.2/cyan,0.6/blue,1/violet}
{
\pgfmathsetmacro{\pos}{\s};
\begin{scope}
\clip(0,0) circle (1cm);
\draw[black, rotate=45,dashed](-1,\pos) -- (1,\pos);
\end{scope}
\begin{scope}
\clip(\globalshift,0) circle (1cm);
\draw[black,dashed](\globalshift-1,\pos) -- (\globalshift+1,\pos);
\end{scope}

\begin{scope}
\clip(0,0) circle (1cm);
\draw[\mycolor, thick,rotate=-45](-1,\pos) -- (1,\pos);
\end{scope}
\begin{scope}
\clip(\globalshift,0) circle (1cm);
\draw[\mycolor, thick](\globalshift+\pos,-1) -- (\globalshift+\pos,+1);
\end{scope}
}
\end{tikzpicture} 
\begin{tikzpicture}[scale=1.2]
\newcommand{\globalshift}{3}
\newcommand{\innerradius}{0.2}
\newcommand{\outerradius}{0.9}
\draw [fill, green!10] (-1, -1) rectangle (\globalshift+1.4,1.7);
\begin{scope}
\draw (0.8,1.4) node[below,black]{$\imgdom$};
\clip (0,-1) rectangle (1,1);

\draw[fill,gray!50] (0,0) circle (1 cm);
\draw (0.7,0) ;

\draw[fill,white] (0,0) circle (\innerradius cm);

\end{scope}

\begin{scope}
\draw[ ->](0,0) -- (0,1.1) node[above]{$x_2$};
\draw[ ->](0,0) -- (1.1,0) node[right]{$x_1$};
\draw (1.5,1.5) node{\large $\gamma$};

\draw[->,bend right=30, ultra thick](2,1.2) edge (1,1.2);
\end{scope}

\draw[fill,gray!50] (\globalshift-1,-0.8+\innerradius) rectangle (\globalshift+1,0.9);
\draw (\globalshift-0.3,1.4) node[left, black,yshift=-0.2cm]{$\gammagroundset$};
\draw[ ->](\globalshift,-0.8) -- (\globalshift,1.1) node[above]{$\tvariable$};
\draw[ ->](\globalshift,-0.8) -- (\globalshift+1.1,-0.8) node[right]{$\rayvariable$};

\foreach \s/\mycolor in {-1/red,-0.6/orange,-0.2/teal,0.2/cyan,0.6/blue,1/violet}
{
\pgfmathsetmacro{\pos}{\s};
\begin{scope}
\clip(0,0) circle (1cm);
\clip (0,-1) rectangle (1,1);
\draw[black,dashed](0,0) circle (\s cm);
\end{scope}

\begin{scope}
\pgfmathsetmacro{\valone}{cos( deg(\s*1.57))};
\pgfmathsetmacro{\valtwo}{sin( deg(\s*1.57))};

\clip(0,0) circle (1cm);
\clip (0,-1) rectangle (1,1);
\draw[\mycolor, thick,rotate=-0](0,0) -- (\valone,\valtwo);
\end{scope}

\begin{scope}
\clip(\globalshift-1,-0.8+\innerradius) rectangle (\globalshift+1,0.9);
\draw[\mycolor, thick](\globalshift+\pos,-1) -- (\globalshift+\pos,+1);
\end{scope}
}

\foreach \s/\mycolor in {-0.6/cyan,0.1/blue,0.9/violet}
{
\pgfmathsetmacro{\pos}{\s};
\begin{scope}
\clip(\globalshift-1,-0.8+\innerradius) rectangle (\globalshift+1,0.9);
\draw[black,dashed](\globalshift-1,\pos) -- (\globalshift+1,\pos);
\end{scope}
}

\draw[fill,green!10](0,0) circle (\innerradius cm);
\draw[ ->](0,0) -- (0,1.1) node[above]{$x_2$};
\draw[ ->](0,0) -- (1.1,0) node[right]{$x_1$};
\end{tikzpicture}

\caption{On the top, an illustration of a generic bijection $\gamma$ between the domain $\imgdom$ in gray on the left, and $\gammagroundset$ on the right.  The colored straight lines on the right (associated with fixed values for $\rayvariable$) are transformed by $\gamma$ into curves covering $\imgdom$. The dashed black lines illustrate $\gamma$ for fixed $\tvariable$. In particular, $\tgroundset(\rayvariable)$ is the intersection of the colored line representing a fixed $\rayvariable$ with $\gammagroundset$, and is depicted for the orange $\rayvariable$.
In the bottom row, analogous illustrations for the parallel beam geometry setting (on the left), and for a fanbeam geometry setting (on the right) are shown.
}
\label{fig_general_illustration_gamma}
\end{figure}

So $\Projind f$ is a single projection, and $\Projind$ can be interpreted as a weighted curve integral operator; for known weight bounds $[c,C]$.
The curves $\gamma(\rayvariable,\cdot)$ (for $\rayvariable \in \raygroundset$) parametrize the integration lines typical for projection methods, e.g., the paths the radiation follows. Here, $\rayvariable \in \raygroundset$ specify which curve (ray) is used, while $\tvariable$ parametrizes the curve itself; see Figure \ref{fig_general_illustration_gamma}.
These curves are straight lines in most practical cases, but other choices are mathematically possible.

In the classical case of parallel beam measurements, $\gamma(\rayvariable,\tvariable)=\rayvariable\vartheta+\tvariable\vartheta^{\rotated}$ with $\vartheta\in \RR^2$ a unit vector and $\vartheta^{\rotated}$ the $90^\circ$ counterclockwise rotated $\vartheta$.
For instance, when $\imgdom$ is the unit ball in $\RR^2$, then $\raygroundset$ is the interval $(-1,1)$, while $\tgroundset(\rayvariable)= (-\sqrt{1-\rayvariable^2},\sqrt{1-\rayvariable^2})$ and $\gammagroundset=\{ (\rayvariable,\tvariable)\ \big | \ \rayvariable \in (-1,1),\  |\tvariable|<\sqrt{1-\rayvariable^2}\}$ is again the unit ball; see Figure \ref{fig_general_illustration_gamma}. %If $\imgdom$ was the anulus with inner radius $0.5$ and outer radius $1$, $\tgroundset(\rayvariable)= (-\sqrt{1-\rayvariable^2},-\sqrt{0.5^2-\rayvariable^2}) \cup (\sqrt{0.5^2-\rayvariable^2}, \sqrt{1-\rayvariable^2})$ for $|\rayvariable|\leq 0.5$ while it remains as before for $|\rayvariable|>0.5$.

In a fanbeam measurement situation, given a beam vertex position $\lambda \in \RR^2$, we consider $\gamma(\rayvariable,\tvariable)=\lambda+\tvariable \vartheta_{\rayvariable}$, where $\rayvariable$ represents an individual ray angle and $\vartheta_{\rayvariable}:=(\cos \rayvariable,\sin\rayvariable)^T\in \RR^2$ denotes the corresponding unit vector. 
If we assume $\lambda=(0,0)^T$ and set $\imgdom= \{x=(x_1,x_2)^T\in \RR^2 \ \big | \ x_1>0 \text{ and } \ 0.2<|x|<1\}$ the right half of an annulus with inner radius $0.2$ and outer radius $1$, we get $\raygroundset=(-\frac{\pi}{2},\frac{\pi}{2})$ and $\tgroundset(\rayvariable)= (0.2,1)$ independent of $\rayvariable$; thus $\gammagroundset=(-\frac{\pi}{2},\frac{\pi}{2}) \times (0.2,1)$; see Figure \ref{fig_general_illustration_gamma}.

% Note that by construction the Jacobi determinant of $\gamma$ on $\raygroundset \times T$ and its inverse on $\imgdom$ remain bounded. Moreover, since $\gamma^{-1}(\Omega)$ is connected, the smallest open set $A$ such that $\imgdom \subset \gamma(A,\overline T)$ is connect, i.e., an interval (namely $\raygroundset$).
The multiplicative factor $\weightfkt$ describes additional physical processes. It will sometimes be more convenient to express $\weightfkt$ as a function of $x\in \imgdom$ (where $x=\gamma(\rayvariable,\tvariable)$). In that case, we use the notation $\altweightfkt(x):=\weightfkt(\gamma^{-1}(x))$. 
 In the classical case of transmission tomography $\weightfkt=1$, but in single photon emission tomography $\weightfkt(\rayvariable,\tvariable)=e^{-\attenuation\tvariable}$ for constant $\attenuation>0$,  which models uniform attenuation. 
 
% In Section \ref{section_recover_known_conditions}, we go into more detail concerning concrete operators that follow this paradigm.

%Similarly to the theory of the Radon transform, these individual projection operators are continuous operators.

We would like to consider single projection operators in a suitable $L^\infty_c$ and $L^2$ context. Lemma \ref{lemma_basic_continuity} will achieve that goal.

\begin{lemma}
\label{lemma_basic_continuity}
Let $\Projind$ be a single projection operator. Then  $\Projind$ maps from $\mathcal{C}^\infty_c(\imgdom)$ into $L^\infty_c(\raygroundset)$, and is  linear. Moreover, there is a constant $c>0$ such that
\begin{equation}\label{equ_basic_continuity}
\|\Projind f\|_{L^{2}(\raygroundset)} \leq c \|f\|_{L^{2}(\imgdom)} \qquad \text{for all $f \in \mathcal{C}^\infty_c(\imgdom)$}.
\end{equation}
\end{lemma}

\begin{proof}
Linearity is a direct consequence of integrals' linearity. 

Given $f\in \mathcal{C}^\infty_c(\imgdom)$ and $\rayvariable\in \raygroundset$, we use the Hölder inequality ($|\int_A g (z)\dd z| \leq \mathcal{L}(A) \sup_{z\in A} |g(z)|$ with $\mathcal{L}(A)$ the Lebesgue measure of the set $A$) to  estimate
\begin{equation}
|[\Projind f] (\rayvariable)| \leq \mathcal{L}(\tgroundset(\rayvariable)) \sup_{\tvariable \in \tgroundset(\rayvariable)} |f(\gamma(\rayvariable,t)) \weightfkt(\rayvariable,t)|.
\end{equation}
Both the integrand and the length of the integration domain are bounded independent of $\rayvariable$; the former since $f$ and $\weightfkt$ are bounded; the latter since $\gammagroundset$ is bounded. 
Thus, $\Projind f$ is uniformly bounded. %\todo{reread}

 Given $f$ with compact support, the set $\widetilde K = \rayinverse(\supp{f})$ is  compactly contained in $\raygroundset$ (being the continuous image of a compact set). For $\rayvariable \in \raygroundset \setminus \widetilde K$, the curve $\gamma(\rayvariable,\cdot)$ does not intersect $\supp{f}$ and consequently $f(\gamma(\rayvariable,t))=0$ for $\rayvariable \in \raygroundset \setminus \widetilde K$ and $\tvariable \in \tgroundset(\rayvariable)$. Looking at Definition \ref{def_projection_curves}, we can conclude $[\Projind f](\rayvariable)=0$ for $\rayvariable\in \raygroundset \setminus \widetilde K$. Therefore, $\supp{\Projind f}\subset \widetilde K$, and thus $\Projind f$ has compact support in $\raygroundset$. Therefore, we established that $\Projind$ maps into $L^\infty_c(\raygroundset)$.

Using the Jensen's Inequality \cite{durrett2019probability} and substituting $x=\gamma(\rayvariable,\tvariable)$, we see
\begin{multline}
\|\Projind f\|_{L^{2}(\raygroundset)}^{2} \overset{\text{def}}{=} \int_{\raygroundset} \left| \int _{\tgroundset(\rayvariable)} f(\gamma(\rayvariable,\tvariable))\weightfkt(\rayvariable,\tvariable)\dd \tvariable \right|^{2} \dd \rayvariable 
\\ 
\overset{\text{Jensen}}{\leq} 
\int_{\raygroundset}\int_{\tgroundset(\rayvariable)} |f(\gamma(\rayvariable,\tvariable))|^2 |\weightfkt(\rayvariable,\tvariable)|^2 \dd{\tvariable} \  \mathcal{L}(\tgroundset(\rayvariable)) \dd \rayvariable
\\
\overset{\text{subst.}}{=}\int_{\imgdom}|f(x)|^2|\altweightfkt(x)|^2 \left|\det\left(\frac{\dd {\gamma^{-1}}}{\dd x}(x)\right) \right| \mathcal{L}(\tgroundset(\rayinverse(x))) \dd x
\\
\leq \sup_{\rayvariable\in \raygroundset} \mathcal{L}(\tgroundset(\rayvariable))
\left\| \det \left(\frac{\dd{\gamma^{-1}}}{\dd {x}}\right)\right\|_{L^\infty} \|\weightfkt\|_{L^\infty}^{2}  \|f\|^{2}_{L^{2}(\imgdom)},
\end{multline}
where we used that $\mathcal{L}(\tgroundset(\rayinverse(x)))\leq \sup_{\rayvariable\in \raygroundset} \mathcal{L}(\tgroundset(\rayvariable))$ is bounded irrespective of $x$ and $\|\weightfkt\|_{L^\infty}=\|\altweightfkt\|_{L^\infty}$ for $\altweightfkt(x)=\weightfkt(\gamma^{-1}(x))$. So \eqref{equ_basic_continuity} holds with the constant $c= \sqrt{\sup_{\rayvariable\in \raygroundset} \mathcal{L}(\tgroundset(\rayvariable)) \left\| \det \left(\frac{\dd{\gamma^{-1}}}{\dd {x}}\right)\right\|_{L^\infty}} \|\weightfkt\|_{L^\infty}$.

%Due to Lemma \ref{lemma_mu_bounded_on_compact}, $L^\infty_c(\raygroundset) \subset L^2_{\mu}(\raygroundset)$, so the proposed extension is possible.
\end{proof}

Note that if $\gamma$ and $\weightfkt$ were $k$-differentiable, we could replace $L^\infty_c$ with $\mathcal{C}^k_c$ in Lemma \ref{lemma_basic_continuity}. In particular, in the case of infinitely differentiable functions, one could replace $L^\infty_c$ with $\mathcal{C}^\infty_c$. We will, however, avoid such assumptions throughout this paper, as they are not necessary for our considerations.

%Now we know the basic mapping properties of single projection operators. These will be essential in our upcoming investigation of \conditionname s, as they inform us what the suitable functional-analytical context will be to look for conditions.

\subsection{The identification of projection pair range conditions}
\label{section_projection_pair_identification_kernels}

\begin{notation}
Subsequently, $\Projind_1$ and $\Projind_2$ denote single projection operators (Definition \ref{def_projection_curves}) for the same $\imgdom$ but individual $\gammagroundset_{\projvariable}$, $\gamma_{\projvariable}$, $\raygroundset_{\projvariable}, \tgroundset_{\projvariable}$, and $\weightfkt_{\projvariable}$ for $\projvariable \in \{1,2\}$.
\end{notation}

\begin{definition}
Given two single projection operators $\Projind_1$ and $\Projind_2$, we define the projection pair operator 
\begin{align}
&\Proj=(\Projind_1,\Projind_2)\colon \mathcal{C}^\infty_c(\imgdom) \to L^\infty_c(\raygroundset_1)\times L^\infty_c(\raygroundset_2)
\\
&\text{with }\ \Proj f =(\Projind_1 f,\Projind_2 f), \notag
\end{align}
where $\Projind_1$ and $\Projind_2$ are given by Equation \eqref{equ_def_single_proj_operator}.
 
\end{definition}
We assume a fixed projection pair operator $\Proj=(\Projind_1,\Projind_2)$ throughout Section \ref{section_projection_pair_identification_kernels}.

We search for a relationship between $\Projind_1 f$ and $\Projind_2 f$ 
via two functionals $F_1$ and $F_2$ which, when acting on $\Projind_1$ and $\Projind_2$, give 
\begin{equation}\label{equ_def_most_general_range_condition}
F_1(\Projind_1 f) = F_2(\Projind_2 f) \qquad \text{for all }f\in \mathcal{C}^\infty_c(\imgdom). 
\end{equation} 
When $F_1$ and $F_2$ are sufficiently regular, checking \eqref{equ_def_most_general_range_condition} is equivalent (using the Riesz representation \cite[ch.6 section 4, part 8]{royden2010real}) to the total mass in the two projections being equal under suitable weights $\Vone$ and $\Vtwo$ on $\raygroundset_1,\raygroundset_2$, respectively; see \eqref{equ_general_zero_order_consistency} below.

Thus,  for all $g=(g_1,g_2) \in \rg{\Proj}:=\{\Proj f \in L^\infty_c(\raygroundset_1)\times L^\infty_c(\raygroundset_2) \ \big | \ f \in \mathcal{C}^\infty_c(\imgdom)\}$, we have $\langle(g_1,g_2),(\Vone,-\Vtwo)\rangle=0$ and therefore obtained a necessary range condition. This inner product form suggests the pair $(\Vone,-\Vtwo)$ being (in some sense) in the orthogonal complement of the range.
We associate these $(\Vone,\Vtwo)$ with the projection pair operator $\Proj=(\Projind_1,\Projind_2)$ and call these the projection kernels.

%As stated in the introduction, range conditions can be formulated as testing of the two projections against certain functions (which we will call kernels). These functions can naively be understood as being orthogonal to the range.

%As stated in the introduction, vectors $V$ (which we will call kernels) in the orthogonal complement of the range can induce a range condition by testing $V$ against data $g$.
%In this section we will investigate the properties of such $V$ in a slightly broader analytical setting.

\begin{definition}
\label{def_LPPRC_kernel}
Given a projection pair operator  $\Proj=(\Projind_1,\Projind_2)$, we say a pair of functions   $\Vone\in L^1_\text{loc}(\raygroundset_1)$  and $\Vtwo\in L^1_\text{loc}(\raygroundset_2)$ are projection kernels  for $\Proj$ when
\begin{multline}\label{equ_general_zero_order_consistency}
\int_{\raygroundset_1} [\Projind_1 f](\rayvariable_1) \Vone(\rayvariable_1) \dd {\rayvariable_1} = \int_{\raygroundset_2} [\Projind_2 f](\rayvariable_2) \Vtwo(\rayvariable_2) \dd {\rayvariable_2}  \qquad \text{for all }f\in \mathcal{C}^\infty_c(\imgdom)
\end{multline}
and $\Vone,\Vtwo$ are not constantly zero (almost everywhere). 

%We speak of a projection pair range condition (\conditionname ) when considering the following range condition induced by a pair of projection kernels $(\Vone,\Vtwo)$ for $\Proj$: Functions $g_1 \in L^\infty_c(\raygroundset_1)$ and $g_2 \in L^\infty_c(\raygroundset_2)$  satisfy
A projection pair range condition (\conditionname) induced by a pair of projection kernels $(\Vone,\Vtwo)$ and applied to $g=(g_1,g_2)\in L^\infty_c(\raygroundset_1)\times L^\infty_c(\raygroundset_2)$ is the requirement that
\begin{equation}\label{equ_dcc_controlable_condition}
\int_{\raygroundset_1} g_1(\rayvariable_1) \Vone(\rayvariable_1) \dd {\rayvariable_1} = \int_{\raygroundset_2} g_2(\rayvariable_2) \Vtwo(\rayvariable_2) \dd {\rayvariable_2}. 
\end{equation}
\end{definition}
If there is an $f \in \mathcal{C}^\infty_c(\imgdom)$ such that $g_1=\Projind_1f$ and $g_2=\Projind_2f$ (i.e., $g\in \rg{\Proj}$), then the corresponding \conditionname{}  is satisfied. Thus, \eqref{equ_dcc_controlable_condition} is a necessary condition for $g$ to be in the range of $\Proj$. %\todo{change kernels to projection kernels throught the paper}

Note that $\Projind_{\projvariable}f$ is bounded and has compact support (see Lemma \ref{lemma_basic_continuity}), and $\Vvariable$ is integrable on that compact support, hence the integrals in \eqref{equ_general_zero_order_consistency} and \eqref{equ_dcc_controlable_condition} are finite. Given some data $g=(g_1,g_2)$, checking the \conditionname{} in \eqref{equ_dcc_controlable_condition} is a way to falsify that $g\in \rg{\Proj}$.

\begin{remark}
\label{remark_orthogonal_conditions}
%\todo{Adapt}
The range condition \eqref{equ_dcc_controlable_condition} can be understood as a dual pairing between $(\Vone,-\Vtwo)\in L^1_\text{loc}(\raygroundset_1)\times L^1_\text{loc}(\raygroundset_2)$ and $(g_1,g_2)\in L^\infty_c(\raygroundset_1)\times L^\infty_c(\raygroundset_2)$. Correspondingly, one can understand projection kernels as inducing a regular distribution that is an annihilator of the range of a projection pair operator $\Proj=(\Projind_1,\Projind_2)$.

There might be annihilators in the dual space of $L^\infty_c(\raygroundset_1)\times L^\infty_c(\raygroundset_2)$ that are not regular, i.e., not locally integrable (and thus not  projection kernels as we understand them), but investigating such annihilators is beyond the scope of this work. 

%It can, however, be analytically advantageous to consider dual pairings in $L^2_{\overline{\mu}}(\sinodom)$ due to its Hilbert space structure.
%That dual pairing takes the form
%\begin{equation}
%\langle g, W \rangle _{L^2_{\overline{\mu}}} =  \int_{\raygroundset_1} g_1(\rayvariable)\Wone(\rayvariable) \mu_1(\rayvariable)\dd \rayvariable+ \int_{\raygroundset_2} g_2(\rayvariable)\Wtwo(\rayvariable) \mu_2(\rayvariable) \dd \rayvariable
%\end{equation}
%for functions $g,W\in L^{2}_{\overline{\mu}}(\sinodom)$. When $V=(V_1,V_2)$ are kernels and $(V_1/\mu_1,V_2/\mu_2) \in L^2_{\overline{\mu}}(\sinodom)$ and we set $W=(W_1,W_2):=(V_1/\mu_1,-V_2/\mu_2)$, the equation \eqref{equ_general_zero_order_consistency} reads as
%$\langle\Proj f, W \rangle_{L^2_{\overline{\mu}}}=0$ and can be extended to hold for all $f\in L^{2}(\imgdom)$. Hence, $W\in \rg{\Proj}^{\rotated}$ and the thus the range condition can be understood as checking orthogonality conditions. We will go deeper into that setting and its consequences in Section \ref{Section_closed_range}.
\end{remark}

Next, we investigate how to identify these projection kernels. To that end, we will require an additional assumption.
The functions $(\rayinverse_1,\rayinverse_2)\colon \imgdom \to \Xgroundset := \{(\rayinverse_1(x),\rayinverse_2(x))\in \raygroundset_1\times \raygroundset_2 \ \big|\ x \in \imgdom\}$ map a point $x\in \imgdom$ to $(\rayvariable_1,\rayvariable_2)=(\rayinverse_1,\rayinverse_2)(x)$ such that $x\in \gamma_1(\rayvariable_1,\cdot)\cap \gamma_2(\rayvariable_2,\cdot)$. The set $\gamma_1(\rayvariable_1,\cdot)\cap \gamma_2(\rayvariable_2,\cdot)$ could contain multiple elements. However, we want $\gamma_1(\rayvariable_1,\cdot)\cap \gamma_2(\rayvariable_2,\cdot)$ to have at most one element $x$ so one can identify each $(\rayvariable_1,\rayvariable_2)$ with a single $x\in \imgdom$. Usually, there are non-intersecting curves, and in this case, $\Xgroundset$ is a strict subset of $\raygroundset_1\times \raygroundset_2$. In any case, $\Xgroundset$ is bounded since $\raygroundset_1$ and $\raygroundset_2$ are. The set $\Xgroundset$ is also connected as the image of the connected set $\imgdom$ under the continuous map $(\rayinverse_1,\rayinverse_2)$.

\begin{assumption}
\label{Assumption_curve_independent}
The map $(\rayinverse_1,\rayinverse_2)\colon \imgdom \to \Xgroundset$ is injective and $\frac{\dd (\rayinverse_1,\rayinverse_2)}{\dd x}$ is regular (i.e., the Jacobi matrix is invertible).
\end{assumption}

\begin{lemma} \label{lemma_X_exists}
Assumption \ref{Assumption_curve_independent} holds if and only if $\Xgroundset$ is open and
there is a continuously differentiable $X\colon \Xgroundset\to \imgdom$ such that $X(\rayinverse_1(x),\rayinverse_2(x))=x$ for $x\in \imgdom$. 
\end{lemma}

For the function $X$ in Lemma \ref{lemma_X_exists}, $X(\rayvariable_1,\rayvariable_2)$ will describe the intersection point of the curves $\gamma_1(\rayvariable_1,\cdot)$ and $\gamma_2(\rayvariable_2,\cdot)$, making $X$ the inverse to the mapping $(\rayinverse_1,\rayinverse_2)$.
%The function $X$ is the inverse mapping of $(\rayinverse_1,\rayinverse_2)$, and $X(\rayvariable_1,\rayvariable_2)$ describes the intersection point of the curves $\gamma_1(\rayvariable_1,\cdot)$ and $\gamma_2(\rayvariable_2,\cdot)$.

\begin{proof}
Let us first assume Assumption \ref{Assumption_curve_independent} holds.
Since $(\rayinverse_1,\rayinverse_2)\colon \imgdom \to \Xgroundset$ is bijective, there is an inverse function $X\colon \Xgroundset \to \imgdom$.
Using the inverse function theorem \cite[Theorem 1.1.7]{hormander2015analysis}, $(\rayinverse_1,\rayinverse_2)(\imgdom)=\Xgroundset$ is open and $X$ is a $\mathcal{C}^1$-function. 

For the converse implication, let us assume there is a function $X$ as described in Lemma \ref{lemma_X_exists}. If there were $x,\tilde x \in \imgdom$ with $(\rayinverse_1,\rayinverse_2)(x)=(\rayinverse_1,\rayinverse_2)(\tilde x)$, then $x = X(\rayinverse_1(x),\rayinverse_2(x))=\tilde x$ implying $x=\tilde x$. Thus, $(\rayinverse_1,\rayinverse_2)$ is injective. Since $X$ is continuously differentiable and we have $x=X((\rayinverse_1,\rayinverse_2)(x))$, we see by differentiation that $\text{Id}_2= \left[\frac{\dd X}{\dd{(\rayvariable_1,\rayvariable_2)}}((\rayinverse_1,\rayinverse_2)(x))\right]\frac{\dd {(\rayinverse_1,\rayinverse_2)}}{\dd x}(x)$ with $\text{Id}_2\in \RR^{2\times 2}$ the identity matrix. Since $\text{Id}_2$ has full rank, so does $\frac{\dd {(\rayinverse_1,\rayinverse_2)}}{\dd x}(x)$.
\end{proof}

\begin{remark} \label{Remark_Assumption_violated_Fanbeam}
Most practical applications employ straight lines as curves. It is well-known that in the Euclidean two-dimensional space, two non-parallel lines meet once (although this intersection point is not necessarily in $\imgdom$). There are zero intersections if they are parallel but do not coincide, and infinitely many intersections if they do coincide.
Therefore, in the case of straight lines, the injectivity in Assumption \ref{Assumption_curve_independent} is relatively easily verified (or contradicted) as one only has to check that there are no parallel intersecting lines (no coinciding lines); see Figure \ref{fig_number_of_intersections}.

%Verifying that the Jacobi matrix is invertible is more technical. Conversely, if $X$ (as described in Lemma \ref{lemma_X_exists}) exists and is differentiable, then $\frac{\dd X}{\dd{(\rayvariable_1,\rayvariable_2)}}(\rayvariable_1,\rayvariable_2)= \left( \frac{\dd (\rayinverse_1,\rayinverse_2)}{\dd x} (X(\rayvariable_1,\rayvariable_2)\right)^{-1}$, and in particular, $\frac{\dd (\rayinverse_1,\rayinverse_2)}{\dd x}$ is regular.

Assumption \ref{Assumption_curve_independent} can be quite strong. For example, it is not fulfilled for a pair of fanbeam projections such that the line joining the two sources crosses $\imgdom$; see Figure \ref{fig_number_of_intersections} b). 
 In Remark \ref{remark_assumption_has_failed}, we discuss how this assumption could be relaxed.

%Remark \ref{remark_assumption_has_failed} will discuss some ramifications of Assumption \ref{Assumption_curve_independent} failing for the concrete example of the fanbeam transform. It shows that, although the theory is not rigorously applicable without the assumption, applying it can still give valuable insights.

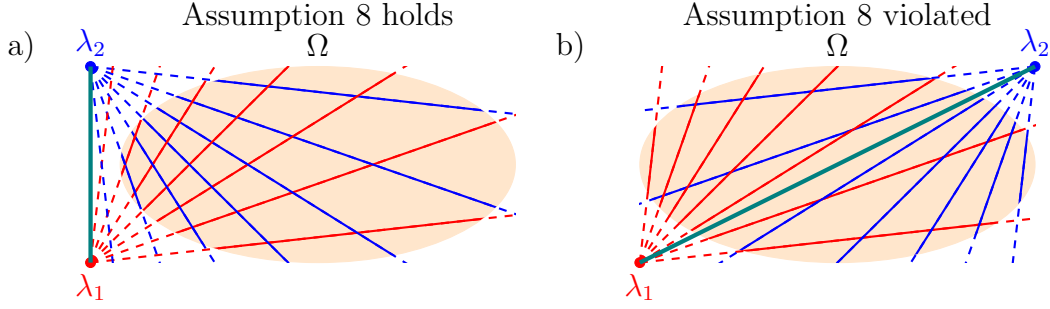
\begin{figure}
\center
\newcommand{\myfiguresize}{1.3}

\begin{tikzpicture}[scale=\myfiguresize]
\pgfmathsetmacro{\myangle}{90};
\pgfmathtruncatemacro\N{7}

\pgfmathsetmacro{\lambdaonex}{-2.3};
\pgfmathsetmacro{\lambdaoney}{-1};

\pgfmathsetmacro{\lambdatwox}{-2.3};
\pgfmathsetmacro{\lambdatwoy}{1};

\draw[](0,1.5) node[] {Assumption \ref{Assumption_curve_independent} holds};
\draw[](-3,1.2) node []{a)};

\draw[red,fill] (\lambdaonex,\lambdaoney) circle(0.05cm) node[below]{$\lambda_1$};
\draw[blue,fill] (\lambdatwox,\lambdatwoy) circle(0.05cm) node[above]{$\lambda_2$};

\draw[fill,orange!20] (0,0) ellipse (2cm and 1cm); \draw[] (0,1) node[above,black]{$\imgdom$};
\foreach \i in {1,...,\N}
{
\pgfmathsetmacro{\directionx}{cos(\i*90/\N-90/\N/2)*10};
\pgfmathsetmacro{\directiony}{sin(\i*90/\N-90/\N/2)*10};

\begin{scope}
\clip[] (0,0) ellipse (2cm and 1cm);

\draw[red,thick] (\lambdaonex,\lambdaoney) -- (\lambdaonex+\directionx,\lambdaoney+\directiony);

\draw[blue,thick] (\lambdatwox,\lambdatwoy) -- (\lambdatwox+\directionx,\lambdatwoy-\directiony);
\end{scope}

\begin{scope}
\clip (-2.5,-1) rectangle (2,1);

\draw[red,thick,dashed] (\lambdaonex,\lambdaoney) -- (\lambdaonex+\directionx,\lambdaoney+\directiony);

\draw[blue,thick, dashed] (\lambdatwox,\lambdatwoy) -- (\lambdatwox+\directionx,\lambdatwoy-\directiony);

\end{scope}
}
\draw[teal, ultra thick] (\lambdaonex,\lambdaoney) -- (\lambdatwox,\lambdatwoy);
\end{tikzpicture}\hfill
\begin{tikzpicture}[scale=\myfiguresize]
\pgfmathsetmacro{\myangle}{90};
\pgfmathtruncatemacro\N{7}

\pgfmathsetmacro{\lambdaonex}{-2};
\pgfmathsetmacro{\lambdaoney}{-1};

\pgfmathsetmacro{\lambdatwox}{2};
\pgfmathsetmacro{\lambdatwoy}{1};
\draw[](0,1.5) node[] {Assumption \ref{Assumption_curve_independent} violated};

\draw[](-2.7,1.2) node []{b)};

\draw[red,fill] (\lambdaonex,\lambdaoney) circle(0.05cm) node[below]{$\lambda_1$};
\draw[blue,fill] (\lambdatwox,\lambdatwoy) circle(0.05cm) node[above]{$\lambda_2$};

\draw[fill,orange!20] (0,0) ellipse (2cm and 1cm); \draw[] (0,1) node[above,black]{$\imgdom$};
\foreach \i in {1,...,\N}
{
\pgfmathsetmacro{\directionx}{cos(\i*90/\N-90/\N/2)*10};
\pgfmathsetmacro{\directiony}{sin(\i*90/\N-90/\N/2)*10};

\begin{scope}
\clip[] (0,0) ellipse (2cm and 1cm);

\draw[red,thick] (\lambdaonex,\lambdaoney) -- (\lambdaonex+\directionx,\lambdaoney+\directiony);

\draw[blue,thick] (\lambdatwox,\lambdatwoy) -- (\lambdatwox-\directionx,\lambdatwoy-\directiony);
\end{scope}

\begin{scope}
\clip (-2,-1) rectangle (2,1);

\draw[red,thick,dashed] (\lambdaonex,\lambdaoney) -- (\lambdaonex+\directionx,\lambdaoney+\directiony);

\draw[blue,thick, dashed] (\lambdatwox,\lambdatwoy) -- (\lambdatwox-\directionx,\lambdatwoy-\directiony);

\end{scope}
}
\draw[teal, ultra thick] (\lambdaonex,\lambdaoney) -- (\lambdatwox,\lambdatwoy);
\end{tikzpicture} 
\caption{Geometric illustration of Assumption \ref{Assumption_curve_independent} for two fanbeam projections as described in Remark \ref{Remark_Assumption_violated_Fanbeam}. It shows in orange the domain $\imgdom$, and in red and blue the corresponding straight lines (curves), solid insight $\imgdom$, and dashed outside. On the left, the fanbeam vertex points are on the left, so that the line connecting  them (in teal) is outside $\imgdom$, while on the right, the connecting line is inside $\imgdom$. In particular, for the $\rayvariable_1$ and $\rayvariable_2$ forming said intersection line, there are infinitely many intersection points.}
\label{fig_number_of_intersections}
\end{figure}
\end{remark}

\begin{lemma}
\label{lemma_connected_separation}
Let Assumption \ref{Assumption_curve_independent}  hold. When $\stripfkt_{\projvariable}\colon \raygroundset_{\projvariable} \to \RR$ for $\projvariable\in \{1,2\}$ are measurable functions with $\stripfkt_1(\rayinverse_1(x))=\stripfkt_2( \rayinverse_2(x))$ for almost all $x\in \imgdom$, then there is a constant $c\in \RR$ with $\stripfkt_1=c$  and $\stripfkt_2=c$ almost everywhere on $\raygroundset_1$ and $\raygroundset_2$, respectively.
\end{lemma}

\begin{remark}
\label{remark_assumption_curves_independent}
In the parallel beam setting with directions $\vartheta_1$ and $\vartheta_2$, the functions $\stripfkt_{\projvariable}\circ\rayinverse_{\projvariable}$ represent parallel-line images, i.e., they are constant along straight lines in directions $\vartheta_1^{\rotated}$ and $\vartheta_2^{\rotated}$; see Figure \ref{fig_stripe_images_parallel}. It seems evident that two such functions can only coincide when they do not show parallel-line structure, which is only the case when they attain the same constant value everywhere, as stated in Lemma \ref{lemma_connected_separation}.
\begin{figure}
\center
\newcommand{\myscale}{1.9}

\begin{tikzpicture}[scale=\myscale]

\pgfextra{\pgfmathsetseed{1}}%
\pgfmathsetmacro{\myangle}{90};
\pgfmathtruncatemacro\N{5}
\pgfmathsetmacro{\stripewidth}{1/\N};

\draw[fill, gray!20] (-1,-1) rectangle (1,1);

\begin{scope}
\clip (-1,-1) rectangle (1,1);
\foreach \s in {0,...,\N}
{
\pgfmathsetmacro{\position}{\s/\N*2-1};

\draw [thick,dashed,magenta,rotate around={\myangle:(0,0)}] (-2,\position) -- (2,\position);
\draw [thick,dashed, blue,rotate around={135:(0,0)}] (-2,\position) -- (2,\position);
}
\pgfmathtruncatemacro\s{3}
\pgfmathsetmacro{\position}{\s/\N*2-1};

\draw [ultra thick,magenta,rotate around={\myangle:(0,0)}] (-2,\position) -- (2,\position);
\draw [ultra thick, blue,rotate around={135:(0,0)}] (-2,\position) -- (2,\position);

\draw[black,fill] (-0.2,-0.08) circle (0.05cm) node[right]{$X(\rayvariable_1,\rayvariable_2)$};
\end{scope}

\end{tikzpicture} \qquad 
\begin{tikzpicture}[scale=\myscale]

\pgfextra{\pgfmathsetseed{1}}%
\pgfmathsetmacro{\myangle}{90};
\pgfmathtruncatemacro\N{30}

\pgfmathsetmacro{\stripewidth}{2/\N};

\clip(-1,-1) rectangle (1,1);

\foreach \s in {0,...,\N}
{
\pgfmathsetmacro{\position}{\s/\N*2-1};

\pgfmathsetmacro{\rng}{rand+1};
\pgfmathsetmacro{\rng}{\rng/2};

\definecolor{mycolor}{gray}{\rng};

\draw [fill,mycolor,rotate around={\myangle:(0,0)}] (-2,\position) rectangle (2,\position+\stripewidth);
}

\end{tikzpicture} \qquad  \begin{tikzpicture}[scale=\myscale]
\pgfextra{\pgfmathsetseed{1}}%
\pgfmathsetmacro{\myangle}{-45};
\pgfmathtruncatemacro\N{30}

\pgfmathtruncatemacro\NN{2*\N}

\pgfmathsetmacro{\stripewidth}{2/\N};

\clip(-1,-1) rectangle (1,1);

\foreach \s in {0,...,\NN}
{
\pgfmathsetmacro{\position}{\s/\N*2-2};

\pgfmathsetmacro{\rng}{rand+1};
\pgfmathsetmacro{\rng}{\rng/2};

\definecolor{mycolor}{gray}{\rng};

\draw [fill,mycolor,rotate around={\myangle:(0,0)}] (-2,\position) rectangle (2,\position+\stripewidth);
}

\end{tikzpicture}
\caption{
Illustration of Lemma \ref{lemma_connected_separation} as described in Remark \ref{remark_assumption_curves_independent} for parallel beam projections with directions $\vartheta_1=(1,0)^T$ and $\vartheta_2=\frac{1}{\sqrt 2} (1,1)^T$ (associated with angles $0^\circ$ and $45^\circ$).
On the left, an illustration of the intersection point of the thick lines $\gamma_1(\rayvariable_1,\cdot)$ and $\gamma_2(\rayvariable_2,\cdot)$ being $X(\rayvariable_1,\rayvariable_2)$.
In the center and on the right, an illustration of typical functions $\stripfkt_1(\rayinverse_1(x))$ and $\stripfkt_2(\rayinverse_2(x))$.
% These take the form of parallel-lines images, illustrating the point of Lemma \ref{lemma_connected_separation} that such functions only coincide when they are constant.
}
\label{fig_stripe_images_parallel}
\end{figure}
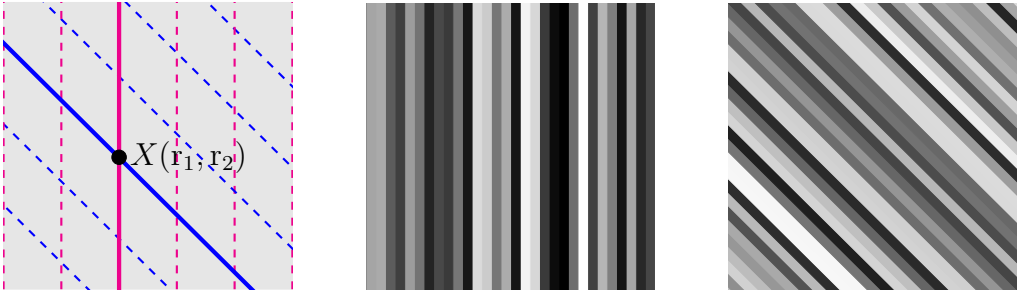
\end{remark}

\begin{proof}[\textbf{Proof of Lemma \ref{lemma_connected_separation}}]
%We start by showing that $\stripfkt_1$ and $\stripfkt_2$ satisfying $\stripfkt_1(\rayinverse_1(x))=\stripfkt_2( \rayinverse_2(x))$ must be locally constant, from which (via connectedness) we will conclude global constantness.

Let $U_1\subset \raygroundset_1$ and $U_2\subset \raygroundset_2$  be open sets such that $U_1\times U_2\subset \Xgroundset$.
Setting $x=X(\rayvariable_1,\rayvariable_2)$ ($X$ as in Lemma \ref{lemma_X_exists}), we see that 
\begin{equation} \label{equ_proof_assumption_lemma}
 \stripfkt_1(\rayvariable_1)= \stripfkt_1(\rayinverse_1(x))= \stripfkt_2(\rayinverse_2(x))= \stripfkt_2(\rayvariable_2) \qquad \text{ for almost all  $(\rayvariable_1,\rayvariable_2)\in U_1\times U_2$}
 \end{equation}
 (since diffeomorphisms map almost everywhere in $\imgdom$ onto almost everywhere in $\Xgroundset$). Since the far-left  side and the far-right side of \eqref{equ_proof_assumption_lemma} depend only on $\rayvariable_1$ or $\rayvariable_2$, respectively, we see that both sides must be constant almost everywhere. Thus, there is a constant $c\in \RR$ with $\stripfkt_1=c$ on $U_1$ and $\stripfkt_2=c$ and $U_2$.
 
Note that $X$ is surjective, so given any $\rayvariable_1^*\in \raygroundset_1$, there is an $x\in \imgdom$ with $\rayinverse_1(x)=\rayvariable_1^*$ and setting $\rayvariable_2^*:=\rayinverse_2(x)$ we see that $(\rayvariable_1^*,\rayvariable_2^*)\in \Xgroundset$. Since $\Xgroundset$ is open (Lemma \ref{lemma_X_exists}), there are neighborhoods $U_1$ and $U_2$ of $\rayvariable_1^*$ and $\rayvariable_2^*$, respectively, such that $U_1\times U_2\subset \Xgroundset$.
As stated above, $\stripfkt_1$ and $\stripfkt_2$ are constant on such $U_1$ and $U_2$. Thus, $\stripfkt_1$ is a locally constant function, i.e., given any point $\rayvariable_1^* \in \raygroundset_1$, there is an open set $U_1\subset \raygroundset_1$ containing $\rayvariable_1^*$ and a constant $c\in \RR$ such that $\stripfkt_1(\rayvariable_1)=c$ on $U_1$.
A locally constant function $\stripfkt_1$ on an
open and connected set $\raygroundset_1$ is globally constant. The same arguments hold for $\stripfkt_2$.
\end{proof}

The subsequent theorem describes a condition for the $\Vone, \Vtwo$ being projection kernels of a \conditionname{} and shows that they are unique in a reasonable sense. Specifically, there is at most one range condition in the form of projection kernels for any projection pair operator.

\begin{theorem}
\label{thm_kernel_condition}
Let $\Proj=(\Projind_1,\Projind_2)$ be a projection pair operator such that $\gamma_{1}$ and $\gamma_{2}$ satisfy Assumption  \ref{Assumption_curve_independent}.
Let $\Vone\in L^1_\text{loc}(\raygroundset_1)$ and $\Vtwo\in L^1_\text{loc}(\raygroundset_2)$. The pair $(\Vone,\Vtwo)$ is a pair of projection kernels of a \conditionname{} for $\Projind_1$ and $\Projind_2$ if and only if $\Vone$ and $\Vtwo$ are almost everywhere non-zero and the `kernel condition'
\begin{equation}
\label{equ_theorem_kernel_conditions}
\frac{\rhoone\big( \gamma_1^{-1}(x)\big)  \left |\det\left( \frac{\dd { \gamma_{1}^{-1 }}}{\dd x}(x)\right)\right|}{\rhotwo\big(\gamma_2^{-1}(x)\big)  \left|\det\left( \frac{ \dd{\gamma_{2}^{-1 }}}{\dd x}(x)\right)\right|} = \frac{\Vtwo(\rayinverse_2(x))}{\Vone(\rayinverse_1(x))} \qquad \text{for a.e. }x\in \imgdom
\end{equation}
is satisfied. Equation  \eqref{equ_theorem_kernel_conditions} is equivalent to
\begin{equation}
\label{equ_theorem_kernel_conditions_2}
\frac{\altrhoone\big(X(\rayvariable_1,\rayvariable_2)\big)  \left |\det\left( \frac{\dd { \gamma_{1}^{-1 }}}{\dd x}(X(\rayvariable_1,\rayvariable_2))\right)\right|}{\altrhotwo\big(X(\rayvariable_1,\rayvariable_2)\big)  \left|\det\left( \frac{\dd{ \gamma_{2}^{-1 }}}{\dd x}(X(\rayvariable_1,\rayvariable_2))\right)\right|} = \frac{\Vtwo(\rayvariable_2)}{\Vone(\rayvariable_1)} \qquad \text{
for a.e. }(\rayvariable_1,\rayvariable_2)\in \Xgroundset
\end{equation}
with $\altrhoi(x)=\rhoi(\gamma^{-1}_{\projvariable}(x))$.
When $(\widetilde \Vone,\widetilde \Vtwo)$ is a second pair of projection kernels, there is a constant $c\in\RR$ with $\widetilde{ \mathrm{V}_\mathrm{i}\, }=c\Vvariable$ a.e. on $\raygroundset_{\projvariable}$ for $\projvariable\in \{1,2\}$. Moreover, $\Vone$ and $\Vtwo$ have the same, constant sign; i.e., they are either both positive or both negative.
\end{theorem}

%Some readers may not be familiar with the measure theoretical concept of "almost every (a.e.)". Roughly speaking, this relates to the fact that changing the functions $\Vone, \Vtwo$ on a set with `length' zero (such as individual points) will not change the integrals in \eqref{equ_general_zero_order_consistency} and is thereby negligible.

\begin{remark}
The kernel conditions \eqref{equ_theorem_kernel_conditions} and \eqref{equ_theorem_kernel_conditions_2} offer a constructive way to determine projection kernels and associated \conditionname s. The left-hand side of the kernel condition \eqref{equ_theorem_kernel_conditions_2} contains known information (the factors $\altrhoi$ and the geometric information $X$ and $\left |\det\left( \frac{\dd { \gamma_{\projvariable}^{-1 }}}{\dd x}(x)\right)\right|$) while the right-hand side contains the projection kernel functions we seek.  %Section \ref{section_recover_known_conditions} will illustrate that this kernel condition indeed yields the \conditionname s for parallel beam and fanbeam transforms whose kernels are already known. 
%Moreover, there is at most one range condition for any projection pair operator in the form of projection kernels.
\end{remark}

\begin{remark} \label{remark-nonexistence-algebraic}
In general, there is no guarantee that the equations \eqref{equ_theorem_kernel_conditions} and \eqref{equ_theorem_kernel_conditions_2} possess solutions $(\Vone,\Vtwo)$. This is illustrated in \eqref{equ_theorem_kernel_conditions_2}, where the right-hand side can be multiplicatively separated into functions dependent on $\rayvariable_1$ and $\rayvariable_2$, respectively. When the left-hand side does not possess the same algebraic structure, projection kernels cannot exist.
A central result of this paper is  that for the exponential fanbeam transform, the left-hand side is not algebraically separable; see Section \ref{section_no_condition_for_exponential_fanbeam}.
\end{remark}

%\begin{remark}
%The uniqueness result of Theorem \ref{thm_kernel_condition} states that the kernel condition uniquely determines the kernels up to multiplication with a constant factor almost everywhere on $\raygroundset_1$ and $\raygroundset_2$, respectively.
%For the $L^2_{\overline{\mu}}$ setting,  this implies that orthogonal complement $\rg{\Proj}^\perp$ contains at most one dimension. Even for the known \conditionname s, their uniqueness appears to be a novel result.
%\end{remark}

\begin{proof}[\textbf{Proof of Theorem \ref{thm_kernel_condition}}]
Let us assume $(\Vone,\Vtwo)$ forms a pair of projection kernels for a \conditionname{} for $\Proj$, i.e., \eqref{equ_general_zero_order_consistency} is satisfied for any $f\in \mathcal{C}^\infty_c(\imgdom)$ and $\Vone,\Vtwo$ are not constantly zero. We aim to show \eqref{equ_theorem_kernel_conditions}. To that end, we substitute the definitions of $\Projind_1$ and $\Projind_2$ into \eqref{equ_general_zero_order_consistency} to see that
\begin{multline} \label{equ_proof_abstract_1}
\int_{\raygroundset_1} \left( \int_{\tgroundset_1(\rayvariable)} f(\gamma_{1}(\rayvariable,\tvariable)) \rhoone(\rayvariable,\tvariable) \dd \tvariable\right)\Vone(\rayvariable) \dd \rayvariable
\\= \int_{\raygroundset_2} \left( \int_{\tgroundset_2(\rayvariable)} f(\gamma_{2}(\rayvariable,\tvariable)) \rhotwo(\rayvariable,\tvariable) \dd \tvariable\right) \Vtwo(\rayvariable) \dd \rayvariable \qquad \text{for all }f\in \mathcal{C}^\infty_c(\imgdom).
\end{multline}
We perform the change of variables $x=x(\rayvariable,\tvariable)=\gamma_{\projvariable}(\rayvariable,\tvariable)$ in the two integrals (with $\projvariable=1$ on the left, $\projvariable=2$ on the right). 
%This substitution is feasible since these curves are diffeomorphisms on $\imgdom$.
Correspondingly, the \conditionname{} \eqref{equ_general_zero_order_consistency} is equivalent to
\begin{multline}\label{equ_proof_abstract_2}
\int_{\imgdom} f(x) \rhoone\big(\gamma_1^{-1}(x)\big) \left|\det\left(\frac{\dd { \gamma_{1}^{-1}}}{\dd x}(x) \right) \right| \Vone ( \rayinverse_1(x)) \dd x
\\
= \int_{\imgdom} f(x) \rhotwo\big(\gamma_2^{-1}(x)\big) \left|\det\left(\frac{\dd { \gamma_{2}^{-1}}}{\dd x}(x) \right) \right| \Vtwo ( \rayinverse_2(x)) \dd x
\end{multline}
for all $f\in \mathcal{C}^\infty_c(\imgdom)$.
By the fundamental lemma of variational calculus \cite{rgen1998calculus}, the integrands must coincide almost everywhere; more precisely,
\begin{multline}
\label{equ_proof_abstract_3}
\rhoone\big(\gamma_1^{-1}(x)\big) \left|\det\left(\frac{\dd {\gamma_{1}^{-1}}  }{\dd x} (x)\right) \right| \Vone ( \rayinverse_1(x)) 
\\ 
=
 \rhotwo\big(\gamma_2^{-1}(x)\big) \left|\det\left(\frac{\dd {\gamma_{2}^{-1}} }{\dd x}(x) \right) \right| \Vtwo ( \rayinverse_2(x))\qquad \text{for a.e. }x\in \imgdom.
\end{multline}
Rearranging \eqref{equ_proof_abstract_3} would yield \eqref{equ_theorem_kernel_conditions}, but one has to be wary to avoid division by zero. This is where Assumption \ref{Assumption_curve_independent} will be crucial.

Due to the assumptions, the functions $\rhotwo(\gamma_2^{-1}(x))$ and the Jacobi determinants $\left|\det\left(\frac{\dd {\gamma_{2}^{-1}} }{\dd x}(x) \right) \right|$ are positive on $\imgdom$; consequently, we can divide by them. We also wish to divide by $\Vone(\rayinverse_1(x))$ to obtain \eqref{equ_theorem_kernel_conditions}. 
To make this feasible, we need to check that it is not zero. 
%Note that if $\Vone(\rayinverse_1(\cdot))=0$ on a non-null set of $\imgdom$, then so must  $\Vtwo(\rayinverse_2(\cdot))$ on the same set (aside from a null set), as otherwise \eqref{equ_proof_abstract_3} cannot be satisfied. We  show that this cannot happen unless both functions $\Vone$ and $\Vtwo$ are constantly zero on $\raygroundset_{1}$ and $\raygroundset_2$, respectively.

We set $\stripfkt_1(\rayvariable_1)=\sign(V_1(\rayvariable_1))$ and $\stripfkt_2(\rayvariable_2)=\sign(V_2(\rayvariable_2))$, where $\sign(z)=1$ if $z>0$, $\sign(z)=-1$ if $z<0$ and $\sign(z)=0$ if $z=0$.
Since $\rhoi$ and the Jacobi determinants are positive, \eqref{equ_proof_abstract_3} implies $\stripfkt_1(\rayinverse_1(x))= \stripfkt_2(\rayinverse_2(x))$, and according to Lemma \ref{lemma_connected_separation}, the $\stripfkt_{\projvariable}(\cdot)=\sign(\Vvariable(\cdot))$ are constant almost everywhere. In particular, if $\Vone$ were zero on a set with non-zero measure, it would be zero almost everywhere, but by definition, projection kernels are not.

Thus, we can reformulate \eqref{equ_proof_abstract_3} by division of $\Vone(\rayinverse_1(x))$ to obtain \eqref{equ_theorem_kernel_conditions}, i.e., projection kernels $\Vone,\Vtwo$ of a \conditionname{}  also satisfy the kernel condition \eqref{equ_theorem_kernel_conditions}.
 To obtain condition \eqref{equ_theorem_kernel_conditions_2}, it remains to replace $x$ with $X(\rayvariable_1,\rayvariable_2)$; a bijective transform that does not impact the equation holding almost everywhere and recall $\altrhoi(x)=\rhoi(\gamma_{\projvariable}^{-1}(x))$.

%All the calculations we made above were equivalent reformulations. Thus, moving in reverse order through those arguments, we see that a pair of (locally integrable, not constantly zero) functions $\Vone,\Vtwo$ satisfying \eqref{equ_theorem_kernel_conditions} yield a pair of kernels of a \conditionname. 
Next, we show the reverse implication, i.e., that the kernel conditions \eqref{equ_theorem_kernel_conditions} imply that $\Vone$ and $\Vtwo$ are projection kernels.
Starting at \eqref{equ_theorem_kernel_conditions}, we multiply by $\Vone$ and $\rhotwo(\gamma_2^{-1}(x)) |\det\left( \frac{ \dd \gamma_2^{-1}}{\dd x}(x)\right)|$ to obtain \eqref{equ_proof_abstract_3}. Multiplying by $f(x)$ and integrating with respect to all $x\in \imgdom$ yields \eqref{equ_proof_abstract_2}, which is equivalent to \eqref{equ_general_zero_order_consistency}, as stated above.

It remains to show that two pairs of projection kernels $(\Vone,\Vtwo)$ and $(\widetilde \Vone,\widetilde \Vtwo)$ only differ by a constant factor.
This additional pair of projection kernels $(\widetilde \Vone,\widetilde \Vtwo)$ also satisfies \eqref{equ_theorem_kernel_conditions}. Therefore, 
\begin{equation} \label{equ_proof_abstract_5}
\frac{\rhoone(\gamma_1^{-1}(x))  \left |\det\left( \frac{\dd {\widetilde \gamma_{1}^{-1 }}}{\dd x}(x)\right)\right|}{\rhotwo(\gamma_2^{-1}(x))  \left|\det\left( \frac{\dd{\widetilde \gamma_{2}^{-1 }}}{\dd x}(x)\right)\right|} =\frac{\widetilde \Vtwo ( \rayinverse_2(x))}{\widetilde \Vone ( \rayinverse_1(x))}=  \frac{\Vtwo ( \rayinverse_2(x))}{\Vone ( \rayinverse_1(x))} \qquad \text{for a.e. }x\in \imgdom
\end{equation}
or equivalently (division is possible as discussed above)
\begin{equation} \label{equ_proof_abstract_6}
\stripfkt_2(\rayinverse_2(x)):=\frac{\widetilde \Vtwo ( \rayinverse_2(x))}{\Vtwo ( \rayinverse_2(x))}=  \frac{\widetilde \Vone ( \rayinverse_1(x))}{\Vone ( \rayinverse_1(x))} =: \stripfkt_1(\rayinverse_1(x))\qquad \text{for a.e. }x\in \imgdom.
\end{equation}
Again, this is the situation $\stripfkt_1(\rayinverse_1(x))=\stripfkt_2( \rayinverse_2(x))$ with $\stripfkt_{\projvariable}(\rayvariable_{\projvariable})= \frac{\widetilde \Vvariable ( \rayvariable_{\projvariable})}{\Vvariable ( \rayvariable_{\projvariable})}$
described in Lemma \ref{lemma_connected_separation}, and therefore, both functions $\stripfkt_{\projvariable}$ are constant on $\raygroundset_1$ and $\raygroundset_2$, respectively, with some common constant $c\in \RR$. This implies $\widetilde \Vone =c \Vone $ and $\widetilde \Vtwo =c \Vtwo $.

%Obviously, multiplication with a constant factor leaves this condition intact. Also mutiplication with a general function $c(x)$ would be possible. But for the reverse consideration we must split $c(x)$ into $c_1(\rayvariable_1(x))$ and $c_2(\rayvariable_2(x))$. Let us assume $c_1$ and $c_2$ are not constant (almost everywhere). Hence there are two non-zero sets $\Xi_1$ and $\Xi\setminus \Xi_1$ in $\Xi$ on which $c_1$ obtains two different values (when $c_1$ is not constant). Naturally $c_2$ must also have two domains to compensate those two possible values, namely $\Xi_2$ and $\Xi\setminus \Xi_2$. But in order to choose the values of $c_1$ and $c_2$ appropriately, it is necessary that $\{\rayvariable_1^{-1}(\Xi_1)\cap \Omega\}\Delta \{\rayvariable_2^{-1}(\Xi_2)\cap \Omega\}$ is a null-set. But under assumption \eqref{} it follows that $\Xi_i\Delta\Xi$ is a null set for $i\in \{1,2\}$. 

\end{proof}

\begin{remark}
\label{remark_assumption_has_failed}
Continuing Remark \ref{Remark_Assumption_violated_Fanbeam}, we want to illustrate some ramifications of Assumption \ref{Assumption_curve_independent} being violated. 
Assumption \ref{Assumption_curve_independent} was a necessity in presenting the entire theory rigorously. Without it, Lemma \ref{lemma_connected_separation} would no longer be true. In that case, there might be divisions zero by zero in \eqref{equ_theorem_kernel_conditions}, as well as a loss of uniqueness in Theorem \ref{thm_kernel_condition}. Moreover, the condition \eqref{equ_theorem_kernel_conditions_2} is nonsensical without the function $X$ (based on Lemma \ref{lemma_X_exists}).

If there is only a single pair of coinciding lines $L:=\gamma_1(\rayvariable_1^*,\cdot)=\gamma_2(\rayvariable_2^*,\cdot)$, as is the case in the fanbeam setting in Figure \ref{fig_number_of_intersections}, one can separate $\imgdom=\imgdom^1\cup \imgdom^2\cup L$ into two connected components plus the said line itself, and analogously $\raygroundset_1=\raygroundset_1^1\cup \raygroundset_1^2\cup\{\rayvariable_1^*\}$, $\raygroundset_2=\raygroundset_2^1\cup \raygroundset_2^2\cup\{\rayvariable_2^*\}$. 
Then, most of the theory can be applied to each component individually, and the coinciding line is only a set of zero measure and can, to some degree, be disregarded.

Then, one can try to find solutions $(\Vone^1,\Vtwo^1)$ to \eqref{equ_theorem_kernel_conditions} on $\imgdom_1$, and analogously $(\Vone^2,\Vtwo^2)$ on $\imgdom_2$. Setting $\Vone$ to $c_1\Vone^1$ on $\raygroundset_1^1$ and $c_2\Vone^2$ on $\raygroundset_1^2$ with constants $c_1$ and $c_2$ (and analogously for $\Vtwo$ with the same constants), \eqref{equ_theorem_kernel_conditions} is satisfied by $(\Vone,\Vtwo)$ almost everywhere on $\imgdom$. If such $(\Vone,\Vtwo)$ are $L^1_{\text{loc}}$ functions, they are still projection kernels as described in Definition \ref{def_projection_curves}. As we will see in Remark \ref{remark_mixed_hilbert_transform}, even if they are not $L^1_\text{loc}$ functions, they might still show some relation to range conditions induced by non-regular distributions. 
\end{remark}

A direct consequence of Theorem \ref{thm_kernel_condition} is that there are no \conditionname s in the case of truncated data, as we discuss next.

\begin{corollary} \label{cor_truncated_no_kernels}
Let the same assumptions as in Theorem \ref{thm_kernel_condition} hold.
Let $\widetilde\raygroundset_1\subset \raygroundset_1$ and $\widetilde\raygroundset_2\subset \raygroundset_2$ be measurable sets with non-zero measure such that $\raygroundset_1\setminus \widetilde \raygroundset_1$ is a set with non-zero measure. 
Given $\projvariable\in \{1,2\}$ and $f\in \mathcal{C}^\infty_c(\imgdom)$, we set $\widetilde \Projind_{\projvariable}f\colon\widetilde \raygroundset_{\projvariable} \to \RR$ with $\widetilde \Projind_{\projvariable}f=[\Projind_{\projvariable}f]_{|\widetilde  \raygroundset_{\projvariable}}$  the restrictions of $\Projind_1$ and $\Projind_2$ to $\widetilde \raygroundset_{1}$ and $\widetilde \raygroundset_{2}$, respectively.
Then, there are no non-zero functions $\widetilde \Vone\in L^1_\text{loc}(\widetilde \raygroundset_1)$, $\widetilde \Vtwo\in L^1_\text{loc}(\widetilde \raygroundset_2)$ such that
\begin{equation}\label{equ_truncated_PPRC}
\int_{\widetilde\raygroundset_1} [\widetilde \Projind_1 f](\rayvariable_1) \widetilde\Vone(\rayvariable_1) \dd {\rayvariable_1} = \int_{\widetilde\raygroundset_2} [\widetilde \Projind_2 f](\rayvariable_2) \widetilde\Vtwo(\rayvariable_2) \dd {\rayvariable_2}  \qquad \text{for all }f\in \mathcal{C}^\infty_c(\imgdom).
\end{equation}
\end{corollary}

Note that \eqref{equ_truncated_PPRC} is an analog to the equation \eqref{equ_general_zero_order_consistency}, defining functions akin to projection kernels for the truncated setting. 

When many projections are taken (and not only two), it is possible that even though each projection is truncated, the collection of gathered data contains the full data of some untruncated projections (of some projection operator), which might yield necessary range conditions. For example, \cite{https://doi.org/10.1118/1.4905161} describes that truncated fanbeam projections with the source on a half circle contain the full data of fanbeam projections along a line segment inside the circle, which in turn yield necessary range conditions (involving data from all the original projections).

\begin{proof}
Let us assume the existence of non-zero functions $\widetilde \Vone$ and $\widetilde \Vtwo$ as described in Corollary \ref{cor_truncated_no_kernels}.
We set $\Vone(\rayvariable_1)= \widetilde \Vone (\rayvariable_1)$ for all $\rayvariable_1\in \widetilde\raygroundset_1$ and $\Vone(\rayvariable_1)=0$ for $\rayvariable_1 \in \raygroundset_1\setminus \widetilde \raygroundset_1$, and analogously for $\Vtwo$. Since $\Projind_{\projvariable}$ coincides with $\widetilde \Projind_{\projvariable}$ on $\raygroundset_{\projvariable}$, \eqref{equ_truncated_PPRC} implies that $\Vone$ and $\Vtwo$ satisfy \eqref{equ_general_zero_order_consistency}. Thus, they are either a pair of projection kernels as in Definition \ref{def_LPPRC_kernel}, or both functions are constantly zero. However, since $\Vone=0$ on the set $\raygroundset_1\setminus \widetilde \raygroundset_1$ with non-zero measure, they are not projection kernels due to Theorem \ref{thm_kernel_condition}. The only remaining possibility is that $\Vone$ and $\Vtwo$ are constantly zero, implying that $\widetilde \Vone$ and $\widetilde \Vtwo$ are zero, contradicting the assumption that they are non-zero.
\end{proof}

When no projection kernels exist, any given $g\in L^2(\raygroundset_1)\times L^2(\raygroundset_2)$ can be approximated arbitrarily well by projections of smooth functions, as we describe next.

\begin{theorem}\label{thm_dense_range_general}
Let $\Proj=(\Projind_1,\Projind_2)$ be a projection pair operator for which no projection kernels $(\Vone,\Vtwo)$ exist. Let $g=(g_1,g_2)\in L^2(\raygroundset_1)\times L^2(\raygroundset_2)$ and fix $\delta>0$. Then, there is an $f\in \mathcal{C}^\infty_c(\imgdom)$ such that $\|g-\Proj f\|_{L^2(\raygroundset_1)\times L^2(\raygroundset_2)} \leq \delta$, where $\|g\|_{L^2(\raygroundset_1)\times L^2(\raygroundset_2)}^2=\|g_1\|^2_{L^2(\raygroundset_1)}+\|g_2\|^2_{L^2(\raygroundset_2)}$. Equivalently, then  $\rg{\Proj}$ is dense in $L^2(\raygroundset_1)\times L^2(\raygroundset_2)$.

\end{theorem}

\begin{proof}
The space $L^2(\raygroundset_1)\times L^2(\raygroundset_2)$ is a Hilbert space induced by the inner product $\langle g,\tilde g\rangle = \int_{\raygroundset_1} g_1\tilde g_1 \dd {\rayvariable_1}+\int_{\raygroundset_2} g_2\tilde g_2 \dd {\rayvariable_2}$. If a vector $W=(W_1,W_2)\in \rg{\Proj}^\perp\setminus \{0\}$ (the orthogonal complement of the range) exists, then $\Vone=W_1$ and $\Vtwo=-W_2$ are projection kernels (as $\langle W,g \rangle =0$ for all $g\in \rg{\Proj}$ implies that $(\Vone,\Vtwo)$ satisfy \eqref{equ_general_zero_order_consistency}). Since no projection kernels exist, also such a $W$ cannot exist, forcing $\rg{\Proj}^\perp=\{0\}$, i.e., $\rg{\Proj}$ has trivial orthogonal complement. It is well known that subspaces with trivial orthogonal complements are dense in Hilbert spaces; i.e., $\rg{\Proj}$ is dense in $L^2(\raygroundset_1)\times L^2(\raygroundset_2)$.
\end{proof}

 In Section \ref{example_surjectivity}, we show a numerical experiment illustrating this density for the exponential fanbeam transform.

\section{Known Range Conditions as PPRCs}
\label{section_recover_known_conditions}

In this section, we illustrate the use of kernel conditions by finding kernels for a mixed parallel-fanbeam projection pair.  These considerations also serve as a stepping stone for the discussion of range conditions for exponential fanbeam projections in Section \ref{section_no_condition_for_exponential_fanbeam}.

\subsection{Classical projections as single projection operators}
\label{section_known_conditions_summary}
%\todo{Rethink notation for this section, maybe use $\phi_i$ for pairs of parallel and $\lambda_i$ for fanbeam projections}
\begin{notation}
We denote the unit circle in $\RR^2$ by $S^1$, and $\vartheta_\theta\in S^1$ denotes the direction vector corresponding to an angle $\theta$, i.e., $\vartheta_{\theta}:=(\cos\theta,\sin\theta)^T\in S^1$.
Moreover, given $\vartheta\in S^1$, the vector $\vartheta^{\rotated}\in S^1$ denotes the (counterclockwise) rotation of $\vartheta$ by $90$ degrees; in particular, $\vartheta^{\rotated}_\theta:=(-\sin\theta,\cos\theta)^T\in S^1$. Similarly, for $z\in \RR^2$ with $z = \|z\|\vartheta$ for some $\vartheta\in S^1$, $z^{\rotated}:=\|z\| \vartheta^{\rotated}$ is the rotation of $z$ by 90 degrees (counterclockwise).
\end{notation}
\begin{definition}
Given an angle $\theta\in [0,2\pi)$, we define the (conventional) parallel beam projection $\Radon_\theta$ (in direction $\theta$)  of a function $f\in \mathcal{C}^\infty_c(\RR^2)$ according to
\begin{align} 
[\Radon_\theta f ](\rayvariable) &:=\int_{\RR} f(\rayvariable \vartheta_{\theta}+\tvariable \vartheta^{\rotated}_{\theta}) \dd \tvariable.
\end{align}
Here, $\theta$ reflects the detector orientation's angle and $\rayvariable\in \RR$ the detector offset.
The (conventional) fanbeam projection $\Fanbeam_\lambda$ for a $\lambda\in \RR^2$ of $f\in \mathcal{C}^\infty_c(\RR^2)$ is defined according to
\begin{align}
[\Fanbeam_\lambda f](\rayvariable)&:=\int_{\RR^+} f(\lambda +\tvariable \vartheta_{\rayvariable}) \dd \tvariable.
\end{align}
Here, $\lambda\in \RR^2$ denotes the source position and the angle $\rayvariable$ reflects the direction in which a ray is emitted from the source.
\end{definition}

%Classically one would define the Radon transforms for $\projvariable \in [0,\pi[$ and the Fanbeam transform for $\projvariable$ along a curve in $\RR^2$. The sets $\projgroundset$ here highlight that our projection operators only contain two projections, as such operators are the focus of our investigation.

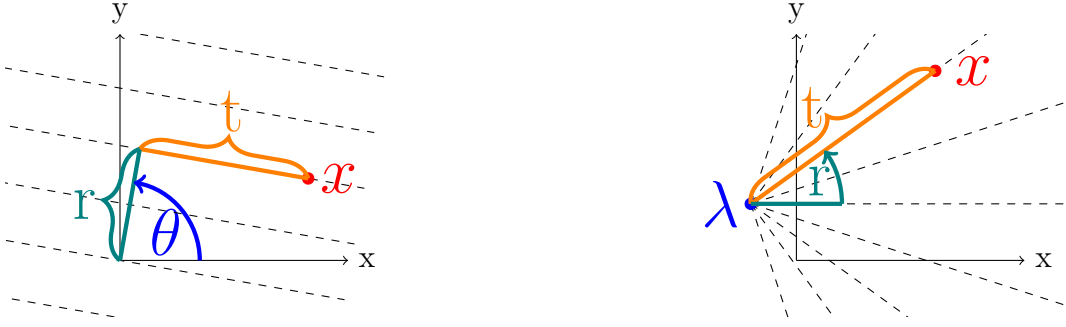
\begin{figure}
\newcommand{\mylinewidth}{ultra thick}
\newcommand{\myfontsize}{\huge}
\begin{tikzpicture}[scale=1.5]
\pgfmathsetmacro{\myangle}{-100};
\pgfmathtruncatemacro\myN{4}

\pgfmathsetmacro{\radius}{0.7};
\pgfmathsetmacro{\x}{-1};
\pgfmathsetmacro{\y}{1.5};

\draw[->] (0,0) -- (0,2) node [above]{y};
\draw[->] (0,0) -- (2,0) node [right]{x};

\clip(-1,-0.5) rectangle (2.4,2);
 \node[blue] at (0.4,0.25)  {\myfontsize $\theta$ };

\foreach \rayvariable in {-\myN,...,\myN}
{
\draw [dashed, rotate around={\myangle:(0,0)}] (\rayvariable/\myN*2,-2) -- (\rayvariable/\myN*2,2);
}

\draw [red,rotate around={\myangle:(0,0)},fill] (\x,\y) circle (0.05cm) node[xshift=0.4cm] {\myfontsize $x$};

\draw [\mylinewidth, orange,rotate around={\myangle:(0,0)}] (-1,0) -- (-1,\y);
\draw [\mylinewidth, teal,rotate around={\myangle:(0,0)}] (0,0) -- (\x,0);

\draw[->,\mylinewidth,blue] (\radius,0) arc (0:180+\myangle:\radius);

\draw [orange,\mylinewidth,decorate,decoration={brace,amplitude=10pt}, rotate around={\myangle:(0,0)}] (-1,0) -- (-1,\y) node [midway,xshift=0.1cm,yshift=0.7cm] 
{ \myfontsize $\tvariable$};

\draw [teal,\mylinewidth,decorate,decoration={brace,amplitude=10pt}, rotate around={\myangle:(0,0)}] (0,0) -- (\x,0) node [midway,xshift=-0.6cm,yshift=0cm] 
{\myfontsize $\rayvariable$};
\end{tikzpicture} \hfill
\begin{tikzpicture}[scale=1.5]
\pgfmathtruncatemacro\myN{5}
\pgfmathtruncatemacro\myNsmall{\myN-1}
\pgfmathsetmacro{\sx}{-0.4};
\pgfmathsetmacro{\sy}{0.5};
\pgfmathsetmacro{\specialangle}{36};
\pgfmathsetmacro{\myt}{2};

\draw[->] (0,0) -- (0,2) node [above]{y};
\draw[->] (0,0) -- (2,0) node [right]{x};

\clip(\sx-0.5,-0.5) rectangle (2.4,2);

\foreach \rayvariable in {-\myNsmall,...,\myNsmall}
{

\pgfmathsetmacro{\myangle}{180*\rayvariable/\myN/2};
\draw [dashed, rotate around={\myangle:(\sx,\sy)}] (\sx,\sy) -- (\sx+3,\sy);
}

\draw [blue,fill] (\sx,\sy) circle (0.05cm) node[xshift=-0.4cm] {\myfontsize $\lambda$};
\draw [red,fill,rotate around={\specialangle:(\sx,\sy)}] (\sx+\myt,\sy) circle (0.05cm) node[xshift=0.5cm] {\myfontsize $x$};
\draw [orange,\mylinewidth,rotate around={\specialangle:(\sx,\sy)}] (\sx,\sy)--(\sx+\myt,\sy);
\draw [orange,\mylinewidth,decorate,decoration={brace,amplitude=10pt}, rotate around={\specialangle:(\sx,\sy)}] (\sx,\sy)--(\sx+\myt,\sy) node [midway,xshift=-0.4cm,yshift=0.4cm] 
{ \myfontsize $\tvariable$};

\newcommand{\myradius}{0.8}
\draw[\mylinewidth,teal] (\sx,\sy)-- (\sx+\myradius,\sy) ;
\draw[->,\mylinewidth,teal] (\sx+\myradius,\sy) arc (0:\specialangle:\myradius);
 \node[teal] at (0.2,0.7)  {\myfontsize $\rayvariable$ };
\end{tikzpicture}
\caption{Illustration of the maps $\gamma_{\theta}^\text{par}$ (see \eqref{equ_definition_curve_parallel}) on the left and $\gamma_{\lambda}^\text{fan}$ (see \eqref{equ_definition_curve_fanbeam}) on the right related to the parallel beam and fanbeam projection curves, respectively. The dashed lines represent the curves induced by various $\rayvariable \in \raygroundset$ (on the left, each $\rayvariable\in \raygroundset$ represents a shift of the ray, while on the right, each $\rayvariable$ is a different angle for the ray) and fixed $\theta$ or $\lambda$, while the orange $\tvariable$ and teal $\rayvariable$ are the coordinates associated with the $x=\gamma(\rayvariable,\tvariable)$ represented by the red dot.}
\label{Figure_illustration_parallel_fanbeam}
\end{figure}

\subsubsection*{Single parallel projection}
We first discuss how $\Radon_\theta$ (for fixed $\theta$) fits the single projection operator setting of this paper.
$\Radon_\theta$ is a line integral operator over straight parallel curves induced by the rotation map
\begin{equation}\label{equ_definition_curve_parallel}
\gamma^\textrm{par}_{\theta}(\rayvariable,\tvariable):= \rayvariable\vartheta_{\theta}+\tvariable \vartheta^{\rotated}_{\theta}
\end{equation}
   weighted by  $\weightfkt(\rayvariable,\tvariable)=\weightfkt(\tvariable)=1$; see Figure \ref{Figure_illustration_parallel_fanbeam}. 
Note that $\gamma_\theta^\text{par}\colon \RR^2\to \RR^2$ is a $\mathcal{C}^\infty$-diffeomorphism.
The inverse function $(\rayinverse_{\theta}^{\text{par}},\tinverse_{\theta}^{\text{par}}):=(\gamma_{\theta}^\text{par})^{-1}$ grants us a unique representation of each point $x\in \RR^2$ via  
\begin{equation}
\label{equ_definition_parametrization_inverse_parallel}
\tinverse_{\theta}^{\text{par}}(x)=x \cdot \vartheta^{\rotated}_{\theta} \qquad \text{ and }\qquad \rayinverse_{\theta}^{\text{par}}(x)=x\cdot \vartheta_{\theta}.
\end{equation}

Given $\imgdom \subset \RR^2$ open, connected  and bounded,
we set $\gammagroundset_{\theta}=[\gamma_{\theta}^\text{par}]^{-1}(\imgdom)$, resulting in  $\raygroundset_{\theta}=\{x\cdot \vartheta_{\theta}\ \big | \ x\in \imgdom\}$ and $\tgroundset_{\theta}(\rayvariable)= \{x \cdot \vartheta^{\rotated}_{\theta}\ \big |\ x\in \imgdom,\ x\cdot \vartheta_{\theta}=\rayvariable\}$.
This $\gammagroundset_{\theta}$ is an open, bounded and connected set, such that the restriction $\gamma_{\theta} := [\gamma_{\theta}^\text{par}] _{\big | \gammagroundset_{\theta}}\colon \gammagroundset_{\theta} \to \imgdom$ is a $\mathcal{C}^1$-diffeomorphism with inverse parametrizations $\tinverse_{\theta}=[\tinverse_{\theta}^\text{par}]_{\big| \imgdom }$ and $\rayinverse_{\theta}=[\rayinverse_{\theta}^\text{par}]_{\big| \imgdom }$ and Jacobi determinant $\left|\det\left(\frac{\dd \gamma_{\theta}^{-1}}{\dd x} \right)\right|=1$.
Thus, $\Radon_\theta $ is a single projection operator (Definition \ref{def_projection_curves}) with corresponding diffeomorphism $\gamma_\theta$ and $\weightfkt_\theta(\rayvariable,\tvariable)=1$ that satisfies the setting of this paper.

\subsubsection*{Single fanbeam projection}
Next, we show how $\Fanbeam_\lambda $ for fixed $\lambda\in \RR^2$ fits into the single projection operator setting of this paper.
To that end, we consider the polar representation centered in $\lambda\in \RR^2$ according to
 \begin{equation}\label{equ_definition_curve_fanbeam}
\gamma^\textrm{fan}_{\lambda}(\rayvariable,\tvariable):= \lambda+\tvariable \vartheta_{\rayvariable} 
 \end{equation}
mapping from $[\theta_0,\theta_0+2\pi) \times \RR^+$ onto $\RR^2\setminus \{\lambda\}$, where $\theta_0\in \RR$ fixed (it will be convenient to shift the standard interval $[0,2\pi)$).  
We see via a basic calculation that the inverse functions $(\rayinverse^\text{fan}_{\lambda},\tinverse^\text{fan}_{\lambda})= (\gamma_{\lambda}^{\text{fan}})^{-1}$ are
\begin{equation} \label{equ_fanbeam_basic_geometry_parametrization}
\rayinverse_{\lambda}^{\text{fan}}(x)= \text{arg}_{\theta_0}(x-\lambda) \qquad \text{and} \qquad \tinverse_{\lambda}^{\text{fan}}(x)=\|x-\lambda\| \qquad \text{for }x\in \RR^2 \setminus \{\lambda\},
\end{equation}
where $\arg_{\theta_0}(z) \in [\theta_0,\theta_0+2\pi)$ is the angle in $[\theta_0,\theta_0+2\pi)$ associated with the vector $z\in \RR^2\setminus\{0\}$, i.e., $\theta=\arg_{\theta_0}(z)$ if and only if $z=\|z\|\vartheta_\theta$. Note that the restriction $\gamma^\text{fan}_\lambda\colon (\theta_0,\theta_0+2\pi)\times \RR^+ \to M_{\theta_0}(\lambda):=\{x \in \RR^2\setminus \{\lambda\} \ \big| \ \text{arg}_{\theta_0}(x-\lambda)\neq \theta_0\}$ is a $\mathcal{C}^1$-diffeomorphism. (We used $\theta_0$ to avoid the discontinuity of the $\text{arg}$ function.)

We consider an  open, connected and bounded $\imgdom$ with $\overline \imgdom \subset M_{\theta_0}(\lambda)$. 
Then, we set 
\[
\raygroundset_{\lambda}=\{\rayinverse^{\text{fan}}_{\lambda}(x) \ \big| \ x \in \imgdom\}, \qquad\tgroundset_{\lambda}(\rayvariable)= \{ \tinverse^\text{fan}_{\lambda}(x) \ \big|\ x\in \imgdom, \ \rayinverse^\text{fan}_{\lambda}(x)=\rayvariable\}\] and $\gammagroundset_{\lambda} = [\gamma_{\lambda}^\text{fan}]^{-1}(\imgdom)= \{ (\rayvariable,\tvariable) \in \RR^2 \big| \rayvariable \in \raygroundset_{\lambda},\ \tvariable\in \tgroundset_{\lambda}(\rayvariable)\}$. In particular, $ \gamma_{\lambda} = [\gamma_{\lambda}^\text{fan}]_{\big| \gammagroundset_{\lambda}} \colon \gammagroundset_{\lambda} \to \imgdom$ is a $\mathcal{C}^1$-diffeomorphism between open, bounded and connected sets with inverse functions $\tinverse_\lambda=[\tinverse_\lambda^\text{fan}]_{|\imgdom}$ and $\rayinverse_\lambda=[\rayinverse_\lambda^\text{fan}]_{|\imgdom}$.
The Jacobi determinant is uniformly bounded and bounded away from zero since 
\begin{equation} \label{equ_fanbeam_determinante_general}
\left|\det \left( \frac{\dd {\gamma_{\lambda} }}{\dd {(\rayvariable, \tvariable)}}(x)\right) \right|= \tinverse_{\lambda}(x)=\|x-\lambda\|
\end{equation}
and $\imgdom$ is bounded and has a positive distance to $\lambda$.
Hence, we can understand $\Fanbeam_\lambda $ as a single projection operator in accordance with the setting of this paper.

\subsection{PPRCs for mixed parallel-fanbeam projection geometries}
\label{section_appendix_different_projection_geometries}

In the projection settings described above, one can derive known data consistency conditions for two parallel beam projections \cite{WOS:A1960WE64800004} or two fanbeam projections \cite{doi:10.1080/01630568308816147}, respectively, by solving the kernel condition \eqref{equ_theorem_kernel_conditions_2}.
However, the theory is broad enough to handle two different underlying geometries.  To illustrate this, we compare a parallel beam projection with a fanbeam projection in this section.

\subsubsection*{Parallel-fanbeam setting}

Given $\imgdom\subset \RR^2$ open, connected and bounded, fixed $\theta\in [0,2\pi)$ and $\lambda\in \RR^2\setminus \overline{\imgdom}$, we set $\gammagroundset_1=(\gamma_{\theta}^\text{par})^{-1}(\imgdom)$ with $\gamma_{\theta}^\text{par}$ as in \eqref{equ_definition_curve_parallel} and consider the $\mathcal{C}^1$-diffeomorphism $\gamma_1:=[\gamma_{\theta}^\text{par}]_{\big| \gammagroundset_1}\colon \gammagroundset_1 \to \imgdom$. Moreover, we assume that $\overline \imgdom\subset M_{\theta_0}(\lambda)$ for some given $\theta_0$ (as in \eqref{equ_fanbeam_basic_geometry_parametrization}). Setting $\gammagroundset_2=(\gamma_{\lambda}^\text{fan})^{-1}(\imgdom)$ with $\gamma_{\lambda}^\text{fan}$ as in \eqref{equ_definition_curve_fanbeam}, we consider  
the $\mathcal{C}^1$-diffeomorphism $\gamma_2:=[\gamma_{\lambda}^\text{fan}]_{\big|\gammagroundset_2}\colon \gammagroundset_2 \to \imgdom$. 
This way, $\Radon_\theta $ and $\Fanbeam_\lambda $ are single projection operators as discussed above, and we consider the corresponding projection pair operator $\Proj^{\theta,\lambda}=(\Radon_\theta,\Fanbeam_\lambda)$.

\subsubsection*{Assumption \ref{Assumption_curve_independent} for $\Proj^{\theta,\lambda}$}
We assume $(\theta - \overline \raygroundset_2) \not \in \frac{\pi}{2}+  \pi \ZZ$, where $\overline \raygroundset_2$ denotes the closure of $\raygroundset_2$ (the angles of beams emitted from $\lambda$ hitting $\imgdom$) and $\pi\ZZ $ are multiples of $\pi$.
Visually, the line in the fanbeam setting with the same direction as the parallel projection does not cut through $\overline \Omega$ (or the parallel beams through $\imgdom$ do not hit $\lambda$); see Figure \ref{fig_mixed_proj}. 
This guarantees that no two straight lines coincide between the two projections, thereby ensuring Assumption \ref{Assumption_curve_independent} (see Remark \ref{Remark_Assumption_violated_Fanbeam}).
This can be interpreted as a requirement on the position of the source $\lambda$, the parallel projection direction $\theta$, or the spatial domain $\imgdom$.

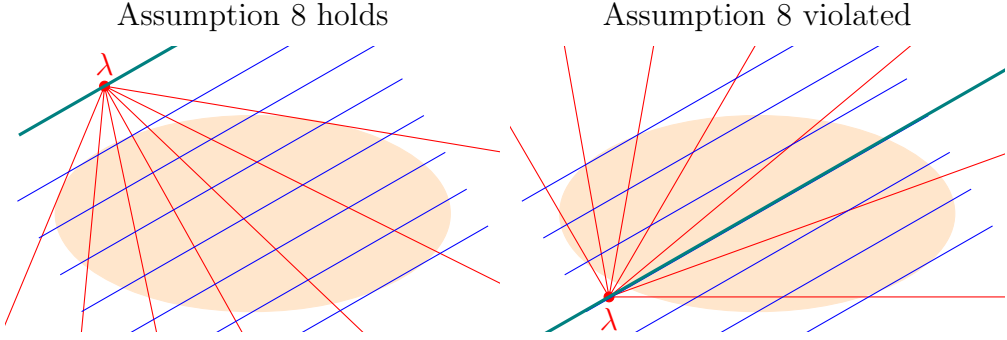
\begin{figure}
\center
\begin{tikzpicture}[scale=1.3]

\pgfmathtruncatemacro\myN{4}
\pgfmathtruncatemacro\myNsmall{\myN-1}
\pgfmathsetmacro{\lambx}{-1.5};
\pgfmathsetmacro{\lamby}{1.3};
\pgfmathsetmacro{\specialangle}{30};
\pgfmathsetmacro{\factor}{0.65};
\pgfmathsetmacro{\anglelower}{165};
\pgfmathsetmacro{\anglehigher}{268};

\draw[](0,2) node[]{Assumption \ref{Assumption_curve_independent} holds};

\draw[fill,orange!20] (0,0) ellipse (2cm and 1cm);
\clip (-2.5,-1.2) rectangle (2.5,1.7);

\draw[fill,red] (-1.5,1.3) circle (0.05 cm) node[above]{$\lambda$};

\foreach \rayvariable in {-\myNsmall,...,\myNsmall}
{

\pgfmathsetmacro{\angle}{\anglelower+(\anglehigher-\anglelower)*\rayvariable/\myNsmall/2+\anglehigher/2};
\pgfmathsetmacro{\directionx}{cos(\angle)};
\pgfmathsetmacro{\directiony}{sin(\angle)};

\draw[red] (\lambx,\lamby) -- (\lambx+10*\directionx,\lamby+10*\directiony);

}

\foreach \rayvariable in {-\myNsmall,...,\myNsmall}
{
\draw[blue, rotate=\specialangle] (-2,\factor*2*\rayvariable/\myNsmall) -- (2,\factor*2*\rayvariable/\myNsmall);
}

\pgfmathsetmacro{\directionx}{cos(\specialangle)};
\pgfmathsetmacro{\directiony}{sin(\specialangle)};
\draw[teal, very thick] (\lambx-\directionx,\lamby-\directiony) -- (\lambx+\directionx,\lamby+\directiony); 
\end{tikzpicture}
\begin{tikzpicture}[scale=1.3]

\pgfmathtruncatemacro\myN{4}
\pgfmathtruncatemacro\myNsmall{\myN-1}
\pgfmathsetmacro{\lambx}{-1.5};
\pgfmathsetmacro{\lamby}{-0.85};
\pgfmathsetmacro{\specialangle}{30};
\pgfmathsetmacro{\factor}{0.65};
\pgfmathsetmacro{\anglelower}{0};
\pgfmathsetmacro{\anglehigher}{120};

\draw[](0,2) node[]{Assumption \ref{Assumption_curve_independent} violated};

\draw[fill,orange!20] (0,0) ellipse (2cm and 1cm);
\clip (-2.5,-1.2) rectangle (2.5,1.7);

\draw[fill,red] (\lambx,\lamby) circle (0.05 cm) node[below]{$\lambda$};

\foreach \rayvariable in {-\myNsmall,...,\myNsmall}
{

\pgfmathsetmacro{\angle}{\anglelower+(\anglehigher-\anglelower)*\rayvariable/\myNsmall/2+\anglehigher/2};
\pgfmathsetmacro{\directionx}{cos(\angle)};
\pgfmathsetmacro{\directiony}{sin(\angle)};

\draw[red] (\lambx,\lamby) -- (\lambx+10*\directionx,\lamby+10*\directiony);

}

\foreach \rayvariable in {-\myNsmall,...,\myNsmall}
{
\draw[blue, rotate=\specialangle] (-2,\factor*2*\rayvariable/\myNsmall) -- (2,\factor*2*\rayvariable/\myNsmall);
}

\pgfmathsetmacro{\directionx}{cos(\specialangle)};
\pgfmathsetmacro{\directiony}{sin(\specialangle)};
\draw[teal, very thick] (\lambx-10*\directionx,\lamby-10*\directiony) -- (\lambx+10*\directionx,\lamby+10*\directiony); 
\end{tikzpicture}
\caption{Illustration of the assumption on the geometry in Section \ref{section_appendix_different_projection_geometries}, with blue parallel lines with angle $\theta$, and red fanbeam lines originating in $\lambda$. On the left, the assumption is satisfied, while on the right, it is violated as the teal line emitted from $\lambda$ coincides with a parallel line, or (equivalently) the parallel lines passing through $\imgdom$ do hit $\lambda$.
}
\label{fig_mixed_proj}
\end{figure}

This assumption is equivalent to the existence of some $\delta>0$ with 
\begin{equation}\label{equ_mixed_projectoin_Omega}
|\vartheta_\theta \cdot \vartheta_{\rayinverse_2(x)}|>\delta \text{ for all }x \in \imgdom \qquad \text{and} \qquad |\vartheta_\theta \cdot \vartheta_{\rayvariable_2}| > \delta \text{ for all }\rayvariable_2 \in \raygroundset_2.
\end{equation}
In turn, this is equivalent to the existence of a constant $\delta>0$ such that
\begin{equation} \label{equ_mixed_projectoin_Omega_2}
|\rayvariable_1-\lambda\cdot \vartheta_\theta|>\delta \qquad \text{for all }\rayvariable_1 \in \raygroundset_1.
\end{equation}

 Given $x\in \imgdom$, there is a representation
\begin{equation}
\label{equ_connection_radon_fanbeam1}
x=\rayinverse_1(x)\vartheta_{\theta}+\tinverse_1(x) \vartheta^{\rotated}_{\theta} = \lambda+ \tinverse_2(x) \vartheta_{\rayinverse_2(x)}.
\end{equation}
We multiply \eqref{equ_connection_radon_fanbeam1} by $\vartheta_{\theta}$ to see that
\begin{equation}
\label{equ_connection_radon_fanbeam2}
\tinverse_2(x)= \frac{\rayinverse_1(x)-\lambda\cdot \vartheta_{\theta}}{\vartheta_{\theta}\cdot \vartheta_{\rayinverse_2(x)}}.
\end{equation}
Note that the denominator is bounded away from zero due to assumption (see \eqref{equ_mixed_projectoin_Omega}).  Combining \eqref{equ_connection_radon_fanbeam1} with \eqref{equ_connection_radon_fanbeam2}, we see that 
 \begin{equation}
 X(\rayvariable_1,\rayvariable_2)=\lambda+ \frac{\rayvariable_1-\lambda\cdot \vartheta_{\theta}}{\vartheta_{\theta}\cdot \vartheta_{\rayvariable_2}} \vartheta_{\rayvariable_2}
\end{equation}
is a $\mathcal{C}^1$-diffeomorphism, i.e., Assumption \ref{Assumption_curve_independent} holds. 

\subsubsection*{Kernel conditions for parallel-fanbeam projections}
The conditions for Theorem \ref{thm_kernel_condition} are fulfilled, and we next aim to determine the projection kernels $(\Vone,\Vtwo)$ for $\Proj^{\theta,\lambda}=(\Radon_\theta,\Fanbeam_\lambda)$. In this case, three of the four terms on the left-hand side of the kernel condition  \eqref{equ_theorem_kernel_conditions} are equal to one. The other term, $\left|\det\left( \frac{\dd { \gamma_{2}^{-1}}   }{\dd x} (x) \right)\right|$, is equal to $\frac{1}{\tau_2(x)}$ (see \eqref{equ_fanbeam_determinante_general}). Therefore,
\begin{equation} \label{equ_mixed_ansatz_dcc}
\tinverse_2(x) \overset{\eqref{equ_theorem_kernel_conditions}}{=} \frac{\Vtwo(\rayinverse_2(x))}{\Vone(\rayinverse_1(x))} \qquad \text{for a.e. }x\in \imgdom.
\end{equation}
Using \eqref{equ_connection_radon_fanbeam2}, the kernel condition \eqref{equ_mixed_ansatz_dcc} possesses the unique (up to multiplication by constant factors) solution
\begin{equation}
\label{mixed_kernel_result}
\Vone(\rayvariable_1) = \frac{1}{\rayvariable_1-\lambda\cdot \vartheta_{\theta}} \qquad \text{ and} 
\qquad 
\Vtwo(\rayvariable_2) = \frac{1}{\vartheta_{\theta}\cdot \vartheta_{\rayvariable_2}}.
\end{equation}
Note that the denominators in \eqref{mixed_kernel_result} do not vanish due to \eqref{equ_mixed_projectoin_Omega} and \eqref{equ_mixed_projectoin_Omega_2}.

This result corresponds to the so-called `parallel-fanbeam Hilbert condition', a special case of the conditions discussed in \cite[Section 4]{Hamaker_divergent_beam_1980} (with $\vartheta_\theta \cdot\vartheta_{r_2}=\cos(\theta-r_2)$); see also \cite[Equation (14)]{5484190}, and \cite{Frederic_Noo_2002,new_data_consistency_fanbeam}.

\begin{remark}\label{remark_mixed_hilbert_transform}
The parallel-fanbeam Hilbert condition holds even if Assumption \ref{Assumption_curve_independent} is violated (i.e., a fanbeam line parallel to the parallel projections passes through $\imgdom$), and can be more useful; see the application to `less than short scans' in \cite{Frederic_Noo_2002}. In our notation, one could write that condition as 
\begin{equation}\label{equ_Hilbert_condition_Fanbeam}
\lim_{\epsilon\to 0}\int_{\raygroundset_1\setminus \left(\lambda\cdot \vartheta_\theta+ [-\epsilon,\epsilon]\right)} \frac{g_1(\rayvariable_1)}{r_1-\lambda\cdot \vartheta_\theta} \dd {\rayvariable_1}
=
\lim_{\varepsilon\to 0}\int_{\raygroundset_2\setminus{[\theta-\varepsilon,\theta+\varepsilon]}} \frac{g_2(\rayvariable_2)}{\vartheta_\theta \cdot \vartheta_{\rayvariable_2}} \dd {\rayvariable_2}.
\end{equation}
Note that (at first glance) \eqref{equ_Hilbert_condition_Fanbeam} looks exactly like the \conditionname{} induced by the kernels found in \eqref{mixed_kernel_result}, but since \eqref{equ_mixed_projectoin_Omega} and \eqref{equ_mixed_projectoin_Omega_2} do no longer hold (and thus, $\Vone$ and $\Vtwo$ are no longer in $L^1_\text{loc}$), we need to introduce symmetric limits (Cauchy Principle Values) in order to ensure the existence of these integrals.

Therefore, \eqref{equ_Hilbert_condition_Fanbeam} is not a \conditionname{} as we understand them in this paper, but induced by the Hilbert transform, a non-regular distribution, making the condition potentially more sensitive to noise unless suitably regularized. 

\end{remark}

\section{Pairwise Range Conditions for the Exponential Fanbeam Transform}
\label{section_no_condition_for_exponential_fanbeam}

The presented theory is applicable to various concrete projection operators and can be used to confirm some known results, as we saw in Section \ref{section_recover_known_conditions}. This raises the question of whether kernel conditions yield \conditionname s for the projection pair operators whose range is yet to be investigated.
To that end, we discuss the exponential fanbeam transform next, finding that it does \textbf{not} possess a \conditionname{}.

\begin{figure}
\center
\newcommand{\mylinewidth}{ultra thick}
\newcommand{\myfontsize}{\large}
\begin{tikzpicture}[scale=1.5]
\pgfmathtruncatemacro\myN{10}
\pgfmathtruncatemacro\myNsmall{\myN-1}
\pgfmathsetmacro{\sx}{-0.4};
\pgfmathsetmacro{\sy}{0.5};
\pgfmathsetmacro{\myangle}{30};

\pgfmathsetmacro{\Deltanorm}{3}
\pgfmathsetmacro{\Deltax}{\Deltanorm*cos(\myangle)};
\pgfmathsetmacro{\Deltay}{\Deltanorm* sin(\myangle)};

\pgfmathsetmacro{\sxx}{\sx+\Deltax};
\pgfmathsetmacro{\syy}{\sy+\Deltay};

\pgfmathsetmacro{\x}{4.2};
\pgfmathsetmacro{\y}{0.1};
\pgfmathsetmacro{\circlewidth}{0.6};

\draw[->] (0,0) -- (0,2) node [above]{y};
\draw[->] (0,0) -- (2,0) node [right]{x};

\draw[dashed] (-1.25,0) --(0,0);

\begin{scope}
\clip(\sx-2,-0.35) rectangle (5,2.5);

\draw[fill,violet!50] (4,0.2) ellipse (1cm and 0.5cm) node [above ,yshift =0.65cm,xshift=0.1cm,violet]{$\imgdom$};

\draw[green,->,ultra thick](\sx-3*\Deltax,\sy-3*\Deltay) -- (\sxx+3*\Deltax,\syy+3*\Deltay); 
 
\draw [black, opacity=0.2,fill, rotate around={\myangle:(\sx,\sy)}] (\sx-4,\sy) rectangle (\sxx+4,\syy-6) ;
\node[] at (4.3,2.3) {\myfontsize $H$};
\end{scope}

\pgfmathsetmacro{\asdffactor}{0.33};
\pgfmathsetmacro{\asdfradius}{0.3};

\pgfmathsetmacro{\asdfx}{\sx-\Deltax*\asdffactor};
\pgfmathsetmacro{\asdfy}{\sy-\Deltay*\asdffactor};
\draw[orange,ultra thick,rotate around={\myangle+90:(\asdfx,\asdfy)},->]  (\asdfx,\asdfy) -- (1+\asdfx,\asdfy) node [yshift=0.35cm,xshift=0.5cm]{$\vartheta_{\theta_0}=\frac{\Delta^{\rotated}_\lambda}{\|\Delta_\lambda\|}$};
\draw[->,orange,thick,] (\asdfx+\asdfradius,0) arc (0:90+\myangle:\asdfradius) node[xshift=.4cm,yshift=-0.2cm]{$\theta_0$};
\draw[->,cyan,thick,] (\asdfx-\asdfradius*\Deltax/\Deltanorm,\asdfy-\asdfradius*\Deltay/\Deltanorm) arc (180+\myangle:90+\myangle:\asdfradius) node[xshift=-0.05cm,yshift=-0.3cm]{$\cdot$};

%\draw[\mylinewidth,violet] (\sx-\Deltax/\Deltanorm*\circlewidth,\sy-\Deltay/\Deltanorm*\circlewidth) arc (180+\myangle:360+\myangle:\circlewidth);

%\draw[\mylinewidth,violet] (\sxx-\Deltax/\Deltanorm*\circlewidth,\syy-\Deltay/\Deltanorm*\circlewidth) arc (180+\myangle:360+\myangle:\circlewidth);

\foreach \rayvariable in {1,...,\myNsmall}
{
\pgfmathsetmacro{\currentangle}{180+\myangle+180*\rayvariable/\myN}
%\draw[->,ultra thick,violet,rotate around={\currentangle:(\sx,\sy)}] (\sx,\sy) -- (\sx+\circlewidth,\sy);
%\draw[->,ultra thick,violet,rotate around={\currentangle:(\sxx,\syy)}] (\sxx,\syy) -- (\sxx+\circlewidth,\syy);
}
%\node[violet] at (\sx,\sy-\circlewidth-0.12){\myfontsize $\Phi$};
%\node[violet] at (\sxx+\circlewidth+0.15,\syy){\myfontsize $\Phi$};
\draw[fill,yellow] (\x,\y) circle (0.05) node [xshift=0.3cm]{\myfontsize $x$};

\draw[teal,\mylinewidth,dashed] (\sx,\sy) -- (\x,\y);
\pgfmathsetmacro{\direcx}{\x-\sx};
\pgfmathsetmacro{\direcy}{\y-\sy};
\pgfmathsetmacro{\mynorm}{sqrt(\direcx*\direcx+\direcy*\direcy)};
\pgfmathsetmacro{\direcx}{\direcx/\mynorm*\circlewidth};
\pgfmathsetmacro{\direcy}{\direcy/\mynorm*\circlewidth};
\draw[teal, ->,\mylinewidth] (\sx,\sy) -- (\sx+\direcx,\sy+\direcy) node[midway,below, xshift=1cm] {\myfontsize $\vartheta_{\rayinverse_1(x)}$};

\draw [brown,\mylinewidth,decorate,decoration={brace,amplitude=10pt}] (\sx,\sy) -- (\x,\y) node [midway,xshift=0.4cm,yshift=0.65cm] {\myfontsize $\tinverse_1(x)$};

\draw[teal,\mylinewidth,dashed] (\sxx,\syy) -- (\x,\y) ;

\pgfmathsetmacro{\direcx}{\x-\sxx};
\pgfmathsetmacro{\direcy}{\y-\syy};
\pgfmathsetmacro{\mynorm}{sqrt(\direcx*\direcx+\direcy*\direcy)};
\pgfmathsetmacro{\direcx}{\direcx/\mynorm*\circlewidth};
\pgfmathsetmacro{\direcy}{\direcy/\mynorm*\circlewidth};
\draw[teal, ->,\mylinewidth] (\sxx,\syy) -- (\sxx+\direcx,\syy+\direcy) node[midway,below, xshift=-0.2cm,yshift=-0.2cm] {\myfontsize $\vartheta_{\rayinverse_2(x)}$};

\draw [brown,\mylinewidth,decorate,decoration={brace,amplitude=10pt}] (\sxx,\syy) -- (\x,\y) node [midway,xshift=0.7cm,yshift=0.5cm] {\myfontsize $\tinverse_2(x)$};
\draw [blue,fill] (\sx,\sy) circle (0.05cm) node[xshift=-0.cm,yshift=0.4cm] {\myfontsize $\lambda_1$};
\draw [blue,fill] (\sxx,\syy) circle (0.05cm) node[yshift=0.5cm] {\myfontsize $\lambda_2$};
\draw[red,->,ultra thick](\sx,\sy) -- (\sxx,\syy) node[midway,above,yshift=0.1cm ] {\myfontsize $\Delta_\lambda $};

\end{tikzpicture}
\caption{Illustration of the fanbeam geometry with two source positions $\lambda_1$ and $\lambda_2$ (blue). The right (gray) half-plane $H$. The vector $\Delta_\lambda $ between $\lambda_1$ and $\lambda_2$ in red (prolonging that ray in green). 
%The violet circles and directions represent the plausible directions $\Phi$. 
The teal dashed lines connecting $\lambda_1$ and $\lambda_2$ to $x\in \imgdom$ represent multiples of $\vartheta_{\rayinverse_{\projvariable}(x)}$, and their distance (in brown) represents $\tinverse_{\projvariable}(x)$.}
\label{fig_fanbeam_geometry_illustration}
\end{figure}
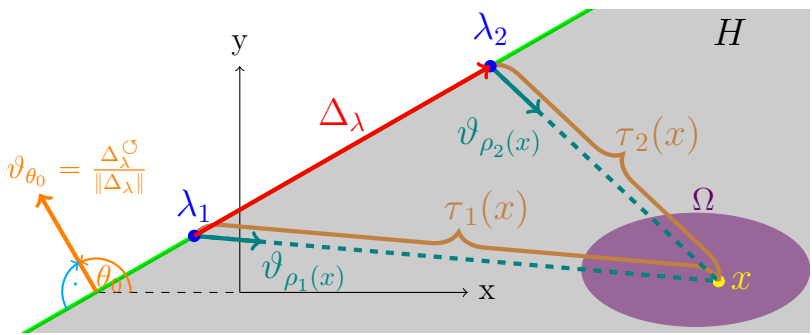

\begin{definition}
Given $\lambda \in \RR^2$, $f \in \mathcal{C}^\infty_c(\imgdom)$, ${\attenuation}\in \RR\setminus\{0\}$ and  $\rayvariable \in [\theta_0,\theta_0+2\pi)$ (for some given fixed $\theta_0$), we define the exponential fanbeam projection according to
\begin{align}
[\ExpoFanbeam_\lambda f](\rayvariable)&:=\int_{\RR^+} f(\lambda +\tvariable \vartheta_{\rayvariable}) e^{{\attenuation} \tvariable} \dd \tvariable.
\end{align}
\end{definition}

For fixed $\lambda\in \RR^2$, $\ExpoFanbeam_\lambda $ is again a single projection operator, with $\gamma_\lambda$ as in Section \ref{section_known_conditions_summary} and $\weightfkt(\rayvariable,\tvariable)=e^{\attenuation \tvariable}$.

In order to ensure Assumption \ref{Assumption_curve_independent}, the following holds throughout this section.
\begin{assumption}\label{Assumption_Fanbeam_geometry}
We consider $\lambda_1,\lambda_2\in \RR^2$ with $\lambda_1\neq \lambda_2$ and we define $\Delta_\lambda: = \lambda_2-\lambda_1$. We assume $\imgdom\subset \RR^2\setminus\{\lambda_1,\lambda_2\}$ open, connected and bounded  such that its closure $\overline \imgdom$ is contained in $H:=\{x \in \RR^2 \ \big |\ (x-\lambda_1)^{\rotated} \cdot \Delta_\lambda  > 0\}$ (the right half-plane of the line induced by $\Delta_\lambda$ through the points $\lambda_1$ and $\lambda_2$; see Figure \ref{fig_fanbeam_geometry_illustration}).  
\end{assumption}

With this assumption, we consider two fanbeam geometries starting in $\lambda_1$ and $\lambda_2$, respectively (as described in Section \ref{section_known_conditions_summary}, $\gamma_1=\gamma_{\lambda_1}^\text{fan}$ and $\gamma_2=\gamma_{\lambda_2}^\text{fan}$).
Then, one can choose $\theta_0=\arg_0(\Delta_\lambda^{\rotated})$, in which case $H\subset M_{\theta_0}(\lambda_1) \cap M_{\theta_0}(\lambda_2)$ (so both projection geometries are well-defined) and $\raygroundset_1,\raygroundset_2 \subset (\theta_0+\frac{\pi}{2},\theta_0+\frac{3\pi}{2})$; see Figure \ref{fig_fanbeam_geometry_illustration}.

In the following, let  $\ExpoFanbeam^{\lambda_1,\lambda_2}=(\ExpoFanbeam_{\lambda_1},\ExpoFanbeam_{\lambda_2})$ denote the projection pair operator whose components are the two respective exponential fanbeam projection operators. 	Recall, $\Xgroundset:=\{ (\rayinverse_1(x),\rayinverse_2(x))\in \raygroundset_1\times \raygroundset_2 \ \big|\  x \in \imgdom\}$, which here will correspond to angle pairs of rays starting at $\lambda_1$ and $\lambda_2$ that intersect inside $\imgdom$.

\subsection{No PPRC for $\ExpoFanbeam^{\lambda_1,\lambda_2}$}
\label{Section_exponential_fanbeam_theory}
%We next investigate \conditionname s for the pairwise exponential fanbeam transform $\ExpoFanbeam^{\lambda_1,\lambda_2}=(\ExpoFanbeam_{\lambda_1},\ExpoFanbeam_{\lambda_2})$.
Theorem \ref{thm_kernel_condition} not only offers a way to identify projection kernels, but if \eqref{equ_theorem_kernel_conditions_2} does not possess solutions $(\Vone,\Vtwo)$, then there are no projection kernels, leading to dense range as described in Theorem \ref{thm_dense_range_general}. We employ this machinery to show that $\ExpoFanbeam^{\lambda_1,\lambda_2}$ does not possess projection kernels (Theorem \ref{thm_no_LPPRC_for_expo_fanbeam}) and thus has dense range (Theorem \ref{thm_exp_fanbeam_dense_range}). This appears to be the first tomographic projection operator for which such a density result is stated.

\begin{theorem} \label{thm_no_LPPRC_for_expo_fanbeam}
Let $\lambda_1,\lambda_2\in \RR^2$, $\attenuation \neq 0$ and let  Assumption \ref{Assumption_Fanbeam_geometry} hold. Then, there is no pair of projection kernels $(\Vone,\Vtwo)$ for $\ExpoFanbeam^{\lambda_1,\lambda_2}$.
\end{theorem}

Before starting the proof, we state some facts concerning the fanbeam geometry obtained by straightforward calculations and using compactness arguments. Visually, they are quite evident in Figure \ref{fig_fanbeam_geometry_illustration}.
There is a constant $\varepsilon>0$ such that
\begin{equation}\label{equ_fanbeam_angle_product_negative}
\vartheta_{\rayinverse_1(x)}^{\rotated}\cdot \vartheta_{\rayinverse_2(x)} {<}-\varepsilon  \quad \text{for all }x\in \imgdom \quad \text{and} \quad
\vartheta_{\rayvariable_1}^{\rotated}\cdot \vartheta_{\rayvariable_2} {<}-\varepsilon  \quad \text{for all }(\rayvariable_1,\rayvariable_2)\in \Xgroundset
\end{equation}
and
\begin{equation}\label{equ_fanbeam_angle_product_negative_2}
\vartheta_{\rayinverse_{\projvariable}(x)}^{\rotated} \cdot \Delta_\lambda > \varepsilon \quad \text{ for all }x\in \imgdom \quad \text{and} \quad  \vartheta_{\rayvariable_{\projvariable}}^{\rotated} \cdot \Delta_\lambda > \varepsilon \quad \text{for }\rayvariable_{\projvariable}\in \raygroundset_{\projvariable}.
\end{equation}
Note that $x\in H$ (rather than $x\in \imgdom$), if and only if both \eqref{equ_fanbeam_angle_product_negative} and \eqref{equ_fanbeam_angle_product_negative_2} hold with $\varepsilon=0$.
Moreover, reformulating $x=\lambda_1+\tinverse_1(x) \vartheta_{\rayinverse_1}(x) = \lambda_2+\tinverse_2(x) \vartheta_{\rayinverse_2(x)}$, one can calculate
\begin{equation}
\label{equ_Fanbeam_determining_t}
\tinverse_1(x)=
\frac{-\vartheta^{\rotated}_{\rayinverse_2(x)} \cdot \Delta_\lambda }{\vartheta^{\rotated}_{\rayinverse_1(x)} \cdot \vartheta_{\rayinverse_2(x)}} \qquad \text{and} \qquad 
\tinverse_2(x)= \frac{-\vartheta^{\rotated}_{\rayinverse_1(x)} \cdot\Delta_\lambda }{\vartheta^{\rotated}_{\rayinverse_1(x)} \cdot \vartheta_{\rayinverse_2(x)}}
\end{equation}
as well as
\begin{equation}\label{equ_fanbeam_X}
X(\rayvariable_1,\rayvariable_2) = \lambda_1 - \frac{ \vartheta^{\rotated}_{\rayvariable_2}\cdot \Delta_\lambda }{\vartheta^{\rotated}_{\rayvariable_1}\cdot \vartheta_{\rayvariable_2}} \vartheta_{\rayvariable_1} \qquad \text{for }(\rayvariable_1,\rayvariable_2)\in \Xgroundset,
\end{equation}
which is a $\mathcal{C}^1$-diffeomorphism on $\Xgroundset$ with a bounded determinant. Note that $X$ has a diffeomorphic extension $\overline X$ on a neighborhood of $\Xgroundset$ (as a subset of $\RR^2$), and thus $\Xgroundset=\overline X^{-1}(\imgdom)$ is open.
Therefore,   Assumption \ref{Assumption_curve_independent} is satisfied (see Lemma \ref{lemma_X_exists}), and the theory of Section \ref{section_Abstract} is applicable.

\begin{proof}[\textbf{Proof of Theorem \ref{thm_no_LPPRC_for_expo_fanbeam}}]
We assume the existence of a pair of projection kernels $\Vone$ and $\Vtwo$ that are positive on $\raygroundset_1$ and $\raygroundset_2$, respectively. (Recall, projection kernels have the same constant sign as described in Theorem \ref{thm_kernel_condition}.) Using the kernel condition \eqref{equ_theorem_kernel_conditions} and equation \eqref{equ_fanbeam_determinante_general}, $(\Vone,\Vtwo)$ satisfy
 \begin{equation} \label{equ_expo_fanbeam_reduced}
e^{{\attenuation} (\tinverse_1(x)-\tinverse_2(x))} \frac{\tinverse_2(x)}{\tinverse_1(x)} 
 %= \frac{\weightfkt(\tinverse_1(x))  \left|\det\left( \frac{\dd { \gamma_{\projvariable_1}^{-1}}  }{\dd x}(x)\right)\right|}{\weightfkt(\tinverse_2(x))  \left|\det\left( \frac{\dd { \gamma_{\projvariable_2}^{-1}}   }{\dd x} (x) \right)\right|}  
 =
\frac{\Vtwo\left(\rayinverse_2(x)\right)}{ \Vone\left(\rayinverse_1(x)\right)}
\qquad \text{ for a.e. }x\in \imgdom.
 \end{equation}
We will show that \eqref{equ_expo_fanbeam_reduced} cannot hold, contradicting the existence of $(\Vone,\Vtwo)$.

We insert \eqref{equ_Fanbeam_determining_t} into \eqref{equ_expo_fanbeam_reduced}, which reduces to
\begin{equation}\label{equ_proof_expo_fanbeam_reduced_plugged_in}
e^{\frac{{\attenuation}}{\vartheta^{\rotated}_{\rayvariable_1} \cdot \vartheta_{\rayvariable_2}}\left( -\vartheta^{\rotated} _{\rayvariable_2} +\vartheta_{\rayvariable_1}^{\rotated} \right)\cdot \Delta_\lambda }\ \frac{\vartheta_{\rayvariable_1}^{\rotated} \cdot \Delta_\lambda }{\vartheta_{\rayvariable_2}^{\rotated} \cdot \Delta_\lambda } =\frac{\Vtwo\left(\rayvariable_2\right)}{ \Vone\left(\rayvariable_1\right)}  \hfill \qquad \text{ for a.e. }(\rayvariable_1,\rayvariable_2) \in \Xgroundset.
\end{equation}

We set $\widetilde \Vone(\rayvariable_1)= (\vartheta^{\rotated}_{\rayvariable_1} \cdot \Delta_\lambda) \Vone(\rayvariable_1)$ and $\widetilde \Vtwo(\rayvariable_2)= (\vartheta^{\rotated}_{\rayvariable_2} \cdot \Delta_\lambda) \Vtwo(\rayvariable_2)$ (note that $\vartheta_{\rayvariable_i}^{\rotated} \cdot \Delta_\lambda>\varepsilon$; see \eqref{equ_fanbeam_angle_product_negative_2}) and obtain
\begin{equation}
e^{\frac{{\attenuation}}{\vartheta_{\rayvariable_1}^{\rotated} \cdot \vartheta_{\rayvariable_2}}\left( {-\vartheta_{\rayvariable_2}^{\rotated}} +{\vartheta_{\rayvariable_1}^{\rotated} }\right)\cdot \Delta_\lambda } =\frac{\widetilde\Vtwo\left(\rayvariable_2\right)}{ \widetilde\Vone\left(\rayvariable_1\right)}  \hfill \qquad \text{ for a.e. }(\rayvariable_1,\rayvariable_2) \in \Xgroundset.
\end{equation}
Applying the logarithm, this reduces to 
\begin{align} \label{equ_proof_non_separable1}
 \frac{{\attenuation}}{\vartheta^{\rotated}_{\rayvariable_1} \cdot \vartheta_{\rayvariable_2}} {(\vartheta^{\rotated}_{\rayvariable_1}  -\vartheta^{\rotated}_{\rayvariable_2}) \cdot \Delta_\lambda } 
&= g_{2}(\rayvariable_2)-g_{1}(\rayvariable_1)
\qquad \text{ for a.e. }(\rayvariable_1,\rayvariable_2) \in \Xgroundset
\end{align}
for the real-valued functions ${g_{1}:=\log(\widetilde \Vone)}$ and ${g_{2}:=\log(\widetilde \Vtwo)}$.
(Note that according to Theorem \ref{thm_kernel_condition}, the signs of the projection kernels are constant, so we can assume that $\widetilde \Vtwo$ and $\widetilde \Vone$ are positive.)

Subtraction of the identity \eqref{equ_proof_non_separable1} with $\widetilde \rayvariable_1 \in \raygroundset_1$ from the same identity with $\rayvariable_1 \in \raygroundset_1$ (i.e., two different arguments for the first variable) such that $(\rayvariable_1,\rayvariable_2),(\widetilde \rayvariable_1,\rayvariable_2)\in \Xgroundset$ yields
\begin{equation}
 \label{equ_proof_non_separable1_1}
F(\rayvariable_1,\tilde \rayvariable_1,\rayvariable_2):={\attenuation}\left[\frac{(\vartheta^{\rotated}_{\rayvariable_1}  -\vartheta^{\rotated}_{\rayvariable_2}) }{\vartheta^{\rotated}_{\rayvariable_1} \cdot \vartheta_{\rayvariable_2}} -\frac{(\vartheta^{\rotated}_{\widetilde \rayvariable_1}  -\vartheta^{\rotated}_{\rayvariable_2})}{ \vartheta^{\rotated}_{\widetilde\rayvariable_1} \cdot \vartheta_{\rayvariable_2}} \right] \cdot \Delta_\lambda  = g_{1}(\widetilde \rayvariable_1)- g_{1}(\rayvariable_1)
\end{equation}
for almost every $\rayvariable_1,\widetilde \rayvariable_1\in \raygroundset_1$ and $\rayvariable_2\in \raygroundset_2$ such that $(\rayvariable_1,\rayvariable_2),(\widetilde \rayvariable_1,\rayvariable_2)\in \Xgroundset$.
Note that the right-hand side of \eqref{equ_proof_non_separable1_1} does not depend on $\rayvariable_2$ (almost everywhere); therefore, neither should the left-hand side.  We set,
\begin{multline} \label{equ_proof_non_separable3}
G(\rayvariable_1,\widetilde \rayvariable_1,\rayvariable_2,\widetilde \rayvariable_2):=F(\rayvariable_1,\widetilde \rayvariable_1,\rayvariable_2)-F(\rayvariable_1,\widetilde \rayvariable_1,\widetilde \rayvariable_2)  
\\
={\attenuation}\left[\frac{(\vartheta^{\rotated}_{\rayvariable_1}  -\vartheta^{\rotated}_{\rayvariable_2}) }{\vartheta^{\rotated}_{\rayvariable_1} \cdot \vartheta_{\rayvariable_2}} -\frac{(\vartheta^{\rotated}_{\widetilde \rayvariable_1}  -\vartheta^{\rotated}_{\rayvariable_2})}{ \vartheta^{\rotated}_{\widetilde\rayvariable_1} \cdot \vartheta_{\rayvariable_2}} -\frac{(\vartheta^{\rotated}_{\rayvariable_1}  -\vartheta^{\rotated}_{\widetilde\rayvariable_2}) }{\vartheta^{\rotated}_{\rayvariable_1} \cdot \vartheta_{\widetilde\rayvariable_2}} +\frac{(\vartheta^{\rotated}_{\widetilde \rayvariable_1}  -\vartheta^{\rotated}_{\widetilde\rayvariable_2})}{ \vartheta^{\rotated}_{\widetilde\rayvariable_1} \cdot \vartheta_{\widetilde\rayvariable_2}} \right] \cdot \Delta_\lambda 
\end{multline}
for  $(\rayvariable_1,\widetilde \rayvariable_1,\rayvariable_2,\widetilde \rayvariable_2)\in \widehat \Xgroundset$ according to
\begin{equation} \label{equ_proof_non_separable4}
\widehat \Xgroundset:=\left\{(\rayvariable_1,\widetilde \rayvariable_1,\rayvariable_2,\widetilde \rayvariable_2) \in \raygroundset_1^2\times \raygroundset_2^2 \ \big |\ \{\rayvariable_{1},\widetilde \rayvariable_{1}\}\times \{\rayvariable_{2},\widetilde \rayvariable_{2}\} \subset \Xgroundset \right\}.
\end{equation} 
The set $\widehat \Xgroundset$ consists of the tuples such that all four combinations of angles correspond to intersection points in $\imgdom$.

Note that $\widehat \Xgroundset$ is again an open set as we describe next.
Let $\hat \rayvariable=(\rayvariable_1,\widetilde \rayvariable_1,\rayvariable_2,\widetilde \rayvariable_2)\in \widehat \Xgroundset$ and $\widehat s=(s_1,\widetilde s_1,s_2,\widetilde s_2)\in \raygroundset_1^2\times \raygroundset_2^2$. Since $\Xgroundset$ is open and $(\rayvariable_1,\rayvariable_2)\in \Xgroundset$, if $|\rayvariable_1 - s_1|$ and $|\rayvariable_2 - s_2|$ are sufficiently small, we have  $(s_1,s_2) \in \Xgroundset$. Analogous considerations hold for the three other combinations; thus $\widehat s\in \widehat \Xgroundset$ if $|\widehat \rayvariable-\widehat s|$ is sufficiently small (for each entry), implying $\widehat \Xgroundset$ is open.

In particular, \eqref{equ_proof_non_separable1_1} implies
\begin{equation} \label{equ_proof_non_separable5}
G(\rayvariable_1,\widetilde \rayvariable_1,\rayvariable_2,\widetilde \rayvariable_2)=0 \qquad\text{for a.e. }(\rayvariable_1,\widetilde \rayvariable_1,\rayvariable_2,\widetilde \rayvariable_2)\in \widehat \Xgroundset.
\end{equation}

The function $G$ (as in the second line of \eqref{equ_proof_non_separable3}) is analytic on $\widehat \Xgroundset$ since $\rayvariable \mapsto \vartheta_{\rayvariable}$ and $\rayvariable \mapsto \vartheta^{\rotated}_{\rayvariable}$ consist of cosine and sine functions that are analytic; all other operations are additions, subtractions, multiplications, and divisions by non-zero (see \eqref{equ_fanbeam_angle_product_negative}), hence, the composition of these functions is also analytic; see \cite{Freitag2011} and \cite{gunning2022analytic}.

Note that $x\in H$ is equivalent to $\rayinverse_1(x)>\rayinverse_2(x)$ (since they are in $(\theta_0+\frac{\pi}{2},\theta_0+\frac{3\pi}{2})$); see Figure \ref{fig_fanbeam_geometry_illustration}. Thus, we set 
\begin{align}
\widetilde \Xgroundset := \Big\{(\rayvariable_1,\widetilde \rayvariable_1,\rayvariable_2,\widetilde \rayvariable_2) \in \left(\theta_0+\frac{\pi}{2},\theta_0+\frac{3\pi}{2}\right)^4 \ \Big |\ & \nu_1>\nu_2  \text{ for all}
\\
& (\nu_1,\nu_2) \in \{\rayvariable_1,\widetilde \rayvariable_1\} \times \{\rayvariable_2,\widetilde \rayvariable_2\} \Big\}. \notag
\end{align}
In particular, $\widetilde \Xgroundset$ is open and connected and $\widehat \Xgroundset \subset \widetilde \Xgroundset$.

Since the denominators in \eqref{equ_proof_non_separable3} are strictly negative on $\widetilde \Xgroundset$, we can extend $G$ (using the same formulation as in \eqref{equ_proof_non_separable3}) to the open and connected set $\widetilde \Xgroundset$. This extension on $\widetilde \Xgroundset$ remains an analytic function (for the same reasons it was analytic on $\widehat \Xgroundset$). 

An analytic function on the connected set $\widetilde \Xgroundset$, which is zero on an  open set $\widehat\Xgroundset\subset \widetilde \Xgroundset$, is zero everywhere on $\widetilde \Xgroundset$. 
If \eqref{equ_proof_non_separable5} were true, then $G=0$ on $\widetilde \Xgroundset$.
Thus, it is sufficient to find a single tuple $(\rayvariable_1^*,\widetilde \rayvariable_1^*,\rayvariable_2^*,\widetilde \rayvariable_2^*) \in \widetilde \Xgroundset$ with $G(\rayvariable_1^*,\widetilde \rayvariable_1^*,\rayvariable_2^*,\widetilde \rayvariable_2^*)\neq 0$, to contradict \eqref{equ_proof_non_separable5} and the existence of $\Vone$ and $\Vtwo$.
We set
\begin{equation}\label{equ_no_fanbeam_counterexample_r}
\rayvariable_1^*=\theta_0+ \pi+\frac{\pi}{4},
\qquad 
\widetilde \rayvariable_1^*= \theta_0+ \pi+\frac{\pi}{6},
\qquad 
\rayvariable_2^*=\theta_0+ \pi,
\qquad
\widetilde\rayvariable_2^* = \theta_0+ \pi-\frac{\pi}{6},
\end{equation}
which satisfy $(\rayvariable_1,\widetilde \rayvariable_1,\rayvariable_2,\widetilde \rayvariable_2)\in \widetilde \Xgroundset$, and recall $\Delta_\lambda= - \|\Delta_\lambda\|\vartheta_{\theta_0}^{\rotated}$.

Note that the inner product is rotation invariant, i.e., $ a \cdot b= (A a)\cdot (A b)$ for a rotation matrix $A$. Since $G$ in \eqref{equ_proof_non_separable3} is a function of the inner products $a\cdot b$ and $a^{\rotated} \cdot b$ with $a,b\in \{\vartheta_{\rayvariable_1},\vartheta_{\widetilde \rayvariable_1},\vartheta_{\rayvariable_2},\vartheta_{\widetilde \rayvariable_2},\Delta_\lambda \}$, rotating these vectors by $-\theta_0$, we can assume without loss of generality $\theta_0=0$.
Straightforward computation shows 
\begin{equation}
G(\rayvariable_1^*,\widetilde \rayvariable_1^*,\rayvariable_2^*,\widetilde \rayvariable_2^*)=(-8-\sqrt 2+4 \sqrt 3+\sqrt 6) \attenuation \|\Delta_\lambda\| \approx-0.0365 \attenuation \|\Delta_\lambda\|\neq 0,
\end{equation}
where we used that $\attenuation \neq 0$ and $\|\Delta_\lambda\|\neq 0$.
This contradicts that $G=0$ on $\widehat\Xgroundset$ (as in \eqref{equ_proof_non_separable5}); hence, $G$ cannot be an analytic function that vanishes almost everywhere on $\widehat \Xgroundset$. Consequently, the initial assumption of the existence of $\Vone$ and $\Vtwo$ cannot be true.
\end{proof}

According to Theorem \ref{thm_no_LPPRC_for_expo_fanbeam}, there are no \conditionname s for the exponential fanbeam transform, and according to Theorem \ref{thm_dense_range_general}, such operators have dense range and any signal is `approximately consistent'. Thus, the following theorem is a direct corollary of Theorem \ref{thm_no_LPPRC_for_expo_fanbeam}.
\begin{theorem}
\label{thm_exp_fanbeam_dense_range}
Let $\lambda_1,\lambda_2\in \RR^2$, let Assumption \ref{Assumption_Fanbeam_geometry} hold and $\attenuation\neq 0$. Let $g\in L^2(\raygroundset_1)\times L^2(\raygroundset_2)$ and $\delta>0$. There is an $f\in \mathcal{C}^\infty_c(\imgdom)$ such that 
\begin{equation}
\|g-\ExpoFanbeam^{\lambda_1,\lambda_2} f\|_{L^2(\raygroundset_1)\times L^2(\raygroundset_2)}\leq \delta.
\end{equation}
 
 %Moreover, there is no $L^2$ continuous function $F\colon L^2(\Omega')\to \RR$ with $F(g)=0$ when $g \in \rg{\Fanbeam_{\attenuation}^{\projgroundset}}$ but $F\not \equiv 0$.
\end{theorem}

\subsection{Numerical approximation of inconceivable data}
\label{Section_expo_fanbeam_numerics}

 \label{example_surjectivity}

The density result described in Theorem \ref{thm_exp_fanbeam_dense_range} might seem surprising as it differs from the situation for projection pair operators with range conditions. Hence, in this section, we illustrate via numerical simulations that the (pairwise) exponential fanbeam transform can attain functions that are unreachable for projection pair operators that admit projection kernels.

\subsubsection*{Inconceivable data}
We first describe what we mean by an unreachable function.
Let $g=(g_1,g_2)$ be two functions (on some open intervals $\raygroundset_1$ and $\raygroundset_2$) with constant $g_1=0$, and $g_2\geq 0$ attains some positive values. If this $g$ were the measurements $\Proj f$ of a function $f$ via a projection pair operator $\Proj$, then $f$ would be invisible from one direction, but only show non-negative values from a different direction.

If $\Proj$ possesses projection kernels $\Vone,\Vtwo$, then said $g$ is not in $\rg{\Proj}$, i.e.,  there is no $f\in \mathcal{C}^\infty_c(\imgdom)$ such that  $g=\Proj f$.
This can be verified using the kernel condition \eqref{equ_dcc_controlable_condition} for $\Proj$. 
Since $g_1=0$, $g_2$ is positive somewhere and non-negative everywhere, and $\Vtwo$ has constant sign (w.l.o.g. is positive), \eqref{equ_dcc_controlable_condition} is not satisfied as
\begin{equation}
\int_{\raygroundset_1} g_1(\rayvariable_1) \Vone(\rayvariable_1) \dd {\rayvariable_1} = 0 \qquad \text{and}  \qquad \int_{\raygroundset_2} g_2(\rayvariable_2) \Vtwo(\rayvariable_2) \dd {\rayvariable_2}>0.
\end{equation} 
 
In particular, given two source positions $\lambda_1,\lambda_2\in \RR^2$ satisfying Assumption \ref{Assumption_Fanbeam_geometry}, the pairwise fanbeam transform $\Fanbeam^{\lambda_1,\lambda_2}:=(\Fanbeam_{\lambda_1},\Fanbeam_{\lambda_2})$ possesses projection kernels $\Vone(\rayvariable_1) = \frac{1}{\vartheta_{\rayvariable_1}^{\rotated} \cdot \Delta_\lambda},\Vtwo(\rayvariable_2)=\frac{1}{\vartheta_{\rayvariable_2}^{\rotated} \cdot \Delta_\lambda}$ (see \cite{doi:10.1080/01630568308816147}, employ Theorem \ref{thm_kernel_condition}, or see \eqref{equ_proof_expo_fanbeam_reduced_plugged_in} with $\attenuation=0$). Consequently, such a $g$ satisfies $g\not \in \rg{\Fanbeam^{\lambda_1,\lambda_2}}$. Because $\Vone$ and $\Vtwo$ are $L^2$ functions (see \eqref{equ_fanbeam_angle_product_negative_2}), the set of $L^2$ functions satisfying \eqref{equ_dcc_controlable_condition} is closed, implying
that there is a positive distance between this $g$ and $\overline {\rg{\Fanbeam^{\lambda_1,\lambda_2}}}$, i.e., there is a $\delta>0$ with $\|g-\Fanbeam^{\lambda_1,\lambda_2} f\|_{L^2(\raygroundset_1)\times L^2(\raygroundset_2)}\geq \delta$ for all $f\in \mathcal{C}^\infty_c(\imgdom)$.

In contrast, for the exponential fanbeam transform $\ExpoFanbeam^{\lambda_1,\lambda_2}$ (with any fixed non-zero $\attenuation$), for each $\delta>0$, there is an $f_{\attenuation}^*\in \mathcal{C}^\infty_c(\imgdom)$ with $\|g-\ExpoFanbeam^{\lambda_1,\lambda_2}f_{\attenuation}^*\|_{L^2(\raygroundset_1)\times L^2(\raygroundset_2)}$ $\leq \delta$  according to Theorem \ref{thm_exp_fanbeam_dense_range}. 
Hence, the described data $g$ are inconceivable for the classical fanbeam transform $\Fanbeam^{\lambda_1,\lambda_2}$, but not for the exponential fanbeam transform $\ExpoFanbeam^{\lambda_1,\lambda_2}$.

\subsubsection*{Description of the numerical experiment}
We use numerical methods to approximate such a function $f_{\attenuation}^*$. %For the sake of comparison, we also try (and will inevitably fail) to approximate a function $f^*$ with $\Fanbeam^{\lambda_1,\lambda_2}f^*=g$.
Concretely, we assume $\lambda_1=(0,80)^T$, $\lambda_2=(-80,0)^T$ and $\imgdom=\{(\bold x, \bold y)\in (-35,35)^2\ \big|\   |\bold y|<\frac 5 {12}( \bold x+80), \ |\bold x|< \frac 5 {12} (80-\bold y)\}$ which satisfy Assumption \ref{Assumption_Fanbeam_geometry}; see Figure \ref{figure_example_surjectivity} a). 
Correspondingly, the relevant angular ranges are  $\raygroundset_1 = (- \frac {\pi} {2}-\alpha_0, - \frac {\pi} {2}+\alpha_0)$ and $\raygroundset_2 = (-\alpha_0,\alpha_0)$ with $\alpha_0 = \arctan\left ( \frac{5}{12}\right)\approx 22.62^\circ$.
We consider, for $\rayvariable \in \raygroundset_2$,
\begin{equation}\label{equ_fanbeam_numerics_definition_g2}
g_2(\rayvariable) = \begin{cases} \frac{1}{4}-9\tan ^2\left(\rayvariable\right) \qquad & \text{if }|\rayvariable|<\arctan(1/6)\approx 9.46^\circ ,
\\
0 & \text{otherwise},
\end{cases}
\end{equation}
and $g_1$ is constantly zero; see Figure \ref{figure_example_surjectivity} b).  We fix $\attenuation=-0.154$, the  attenuation undergone per centimeter of transversed water at 140 keV, the usual operating energy for SPECT.

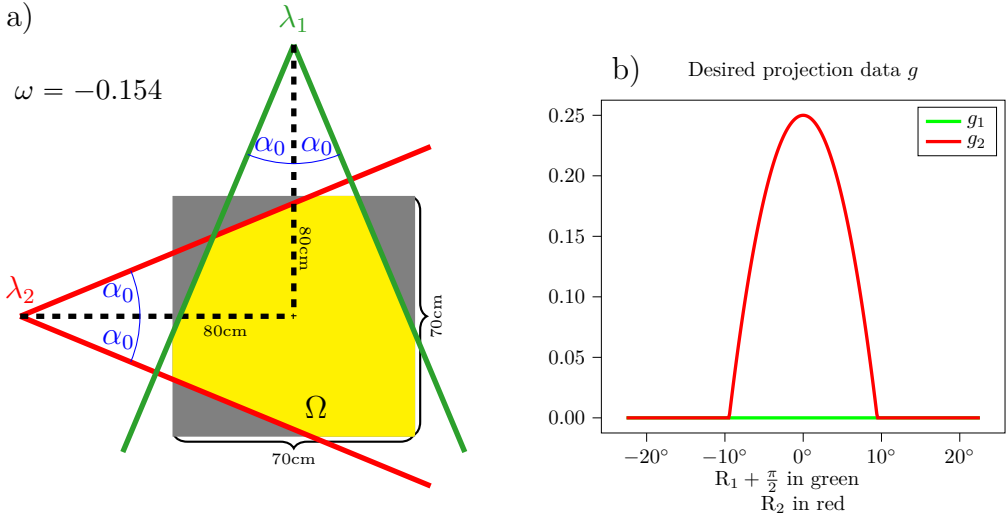
\begin{figure}

\newcommand{\mylinewidth}{2}
\newcommand{\mywidthimg}{0.35}

\definecolor{forestgreen4416044}{RGB}{44,160,44}

\begin{tabular}{lll}
 \begin{tikzpicture}[scale=0.045]
\pgfmathsetmacro{\RE}{80};
\pgfmathsetmacro{\R}{40};
\pgfmathsetmacro{\DW}{100};
\pgfmathsetmacro{\detectordepth}{10};
\pgfmathsetmacro{\factor}{2};
\newcommand{\myfont}{\tiny}

\pgfmathsetmacro{\imagewidth}{70};
\newcommand{\detectorangle}{0}

\clip[] (-90,-60) rectangle (60,110);

\draw [ thick, decorate,decoration={brace,amplitude=5pt}]  (\imagewidth/2,\imagewidth/2) -- (\imagewidth/2,-\imagewidth/2)  node [midway,xshift=0.3cm, rotate=90] { \myfont $70$cm};
\draw [ thick, decorate,decoration={brace,amplitude=5pt}]   (\imagewidth/2,-\imagewidth/2) -- (-\imagewidth/2,-\imagewidth/2)  node [midway,below,yshift=-0.1cm,] { \myfont $70$cm};

\colorlet{shadecolor}{gray!60}
\colorlet{shadecolor2}{orange}
\colorlet{shadecolor3}{magenta}

\draw[fill,gray,very thick]  (-\imagewidth/2,-\imagewidth/2) rectangle (\imagewidth/2,\imagewidth/2) ;

\pgfmathsetmacro{\steepest}{22.61};

\begin{scope}
\pgfmathsetmacro{\phi}{0};
\pgfmathsetmacro{\x}{-\RE*cos(0)};
\pgfmathsetmacro{\phiperp}{\phi+90};
\pgfmathsetmacro{\y}{-\RE*sin(0)};
\pgfmathsetmacro{\upper}{\DW/2};
\pgfmathsetmacro{\lower}{-\DW/2};
\pgfmathsetmacro{\myangle}{};
\clip[] (-100,80) rectangle (100,-50);
\clip[rotate around={\steepest:(\x,\y)}] (\x,\y) rectangle (\R,-200);

\clip[rotate around={-\steepest:(\x,\y)}] (\x,\y) rectangle (\R,200);

\pgfmathsetmacro{\phi}{-90};
\pgfmathsetmacro{\x}{-\RE*cos(\phi)};
\pgfmathsetmacro{\phiperp}{\phi+90};
\pgfmathsetmacro{\y}{-\RE*sin(\phi)};
\pgfmathsetmacro{\upper}{\DW/2};
\pgfmathsetmacro{\lower}{-\DW/2};

\clip[rotate around={\steepest:(\x,\y)}] (\x,\y) rectangle (-200,-200);

\clip[rotate around={-\steepest:(\x,\y)}] (\x,\y) rectangle (200,-200);

\draw[fill,yellow,very thick]  (-\imagewidth/2,-\imagewidth/2) rectangle (\imagewidth/2,\imagewidth/2) ;
\draw[](0,0) node[right,yshift=-1.2cm ]{$\imgdom$};
\end{scope}

%
%\foreach \myphi/\myx/\myy/\mycolor in {-20/0/80/violet, 110/-80/0/cyan,95/-80/0/blue,90/-80/0/olive}
%{
%\pgfmathsetmacro{\phi}{\myphi};
%\pgfmathsetmacro{\stepx}{sin(\phi)};
%\pgfmathsetmacro{\stepy}{-cos(\phi)};
%
%\pgfmathsetmacro{\mystrech}{120};
%
%\draw[\mycolor,line width=\mylinewidth,dotted] (\myx,\myy) -- (\myx+\stepx*\mystrech,\myy+\stepy*\mystrech);
%}

%\foreach \xo/\yo/\phi/\mycolor/\title/\myposition in {-80/0/0/violet/L_2/right,0/80/-90/violet/L_1/below}
%{
%\pgfmathsetmacro{\stepx}{125*cos(\phi)};
%\pgfmathsetmacro{\stepy}{125*sin(\phi)};
%
%\draw[\mycolor,very thick] (\xo,\yo) -- (\xo+\stepx,\yo+\stepy) node[\myposition] {$\title$};
%}

\foreach \phi/\myopacity/\mycolor/\myname/\myxshift/\myyshift in {0/1/red/2/0.9/-0.2,-90/1/forestgreen4416044/1/0.15/-0.9}
{
\pgfmathsetmacro{\x}{-\RE*cos(0)};
\pgfmathsetmacro{\phiperp}{\phi+90};
\pgfmathsetmacro{\y}{-\RE*sin(0)};
\pgfmathsetmacro{\upper}{\DW/2};
\pgfmathsetmacro{\lower}{-\DW/2};
\draw[rotate around={\detectorangle:(0,0)},line width=\mylinewidth,rotate around={\phi:(0,0)},opacity=\myopacity,\mycolor]  (\R,\upper)--(\x,\y) node[above ]{$\lambda_\myname$};
\draw[rotate around={\detectorangle:(0,0)},line width=\mylinewidth,rotate around={\phi:(0,0)},opacity=\myopacity,\mycolor] (\x,\y) -- (\R,\lower);
\draw[dashed,rotate around={\detectorangle:(0,0)},line width=\mylinewidth,rotate around={\phi:(0,0)},opacity=\myopacity] (\x,\y) -- (0,0) node[midway,xshift=\myxshift cm,yshift=\myyshift cm, rotate=\phi ] {\tiny 80cm};
}

\draw(-80,80) node[above]{a)};
\draw(-60,60) node[above]{\small $\attenuation=-0.154$};

\newcommand{\myradius}{35}
\pgfmathsetmacro{\myangle}{atan(5/12)};
\foreach \startx/\starty/\middelangle/\mysign/\xshift/\yshift/\myradiusx/\myradiusy in {-80/0/0/1/-0.25/0/\myradius/0 , -80/0/0/-1/-0.25/0/\myradius/0 , 0/80/-90/1/0/-0.2/0/-\myradius , 0/80/-90/-1/0/-0.2/0/-\myradius}
{
\draw[blue] ( \startx+\myradiusx,\starty+\myradiusy ) arc (\middelangle:\mysign*\myangle+\middelangle:\myradius) node[yshift=-\yshift cm,xshift=\xshift cm,midway]{$\alpha_0$};
}
%
%\draw[,ultra thick](-80,0)--(-104,10);
%\draw[,ultra thick](-80,0)--(-104,-10);
%\draw[ultra thick](-104,10)--(-104,-10);
%\draw[,decorate,decoration={brace,amplitude=5pt}, thick](-104,-10)--(-104,10)node[midway,above,rotate=90]{\tiny $20$cm};
%\draw[ultra thick] (-104,0)--(-80,0);
%\draw[ thick] (-92,0)--(-92,-15)node[below,yshift=-0.0cm]{\tiny $24$cm};

\end{tikzpicture} 
&
\begin{overpic}[height=0.44\textwidth]{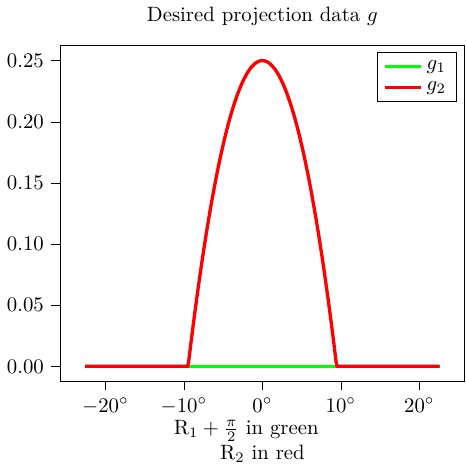}
\put(15,95){\color{black}b)}
\end{overpic}&

\end{tabular}

\caption{Illustration of the setup for Section \ref{example_surjectivity}. In a) the projection geometry is shown, where the relevant image domain $\imgdom$ (yellow) is contained in a 70 cm by 70 cm square, with divergent beam positions $\lambda_1=(0,80)$ and $\lambda_2=(-80,0)$. Figure b) depicts the data functions $g=(g_1,g_2)$ being constantly 0 for $g_1$ and $g_2$ according to \eqref{equ_fanbeam_numerics_definition_g2}. The $\bold x$-axis represents the angles relative to the central lines for $\rayvariable_{\projvariable}\in \raygroundset_{\projvariable}$ and ${\projvariable}\in \{1,2\}$.}
\label{figure_example_surjectivity}
\end{figure}

We aim to find $f^*_{\attenuation}$ satisfying
\begin{equation}\label{equ_expo_fanbeam_example_continuous_equation}
g\essequ\ExpoFanbeam^{\lambda_1,\lambda_2} f_{\attenuation}^*,
\end{equation}
where we cannot write `$=$' as the range is not necessarily closed; with $\essequ$ we mean arbitrarily close.
To that end, we discretize \eqref{equ_expo_fanbeam_example_continuous_equation} by considering the systems of linear equations
\begin{equation}\label{equ_expo_fanbeam_example_discrete_equation}
\widetilde g = \widetilde \ExpoFanbeam^{\lambda_1,\lambda_2}\widetilde f_{\attenuation}
\end{equation}
with $\widetilde g\in \RR^{2\times 400}$ a discretization of $g$ (via pointwise evaluation), a vector (image) $\widetilde f_{\attenuation}\in \RR^{1000\times 1000}$, and an operator (matrix) $\widetilde \ExpoFanbeam^{\lambda_1,\lambda_2}\colon\RR^{1000\times 1000}\to\RR^{2\times 400}$ discretizing the projection pair operator $\ExpoFanbeam^{\lambda_1,\lambda_2}$ (using a pixel-driven approach \cite{doi:10.1137/20M1326635}, which also informs our choice of discretization parameters). 
We employ the `conjugate gradients normal equation' (CGNE) iteration \cite{CG_inversion,Engl96_book_regularization_ip}  to find the minimum norm least-squares solution $\widetilde f_{\attenuation}^*$ to \eqref{equ_expo_fanbeam_example_discrete_equation}.

\subsubsection*{Simulation results}
Figure \ref{Fig_numerics_least_squares} a) shows the least squares solution $\widetilde f_{\attenuation}^*\in \RR^{\myresolution\times \myresolution}$  we obtained after $10000$ CGNE iterations. As can be seen, most activity is located in the center (red and orange `hot' regions), which is used for  $\widetilde \ExpoFanbeam_{\lambda_2} \widetilde f^*_{\attenuation}$ to achieve the non-zero values in $g_2$. Projecting from the left to the right, this area is exactly the one projected onto the positive areas of $g_2$. 
In order to still achieve approximately zero when projecting from $\lambda_1$ (from above), negative values (in the blue 'cold' regions), balance the hot regions, resulting in (attenuation-weighted) line integrals reaching approximately zero. 
Since these colder regions are closer to $\lambda_1$ than the hot ones,  they contribute more (due to attenuation) to the relevant weighted line integrals.
Moreover, the function $\widetilde f^*_{\attenuation}$ has some oscillatory behavior in the upper left parts and attains high values at the boundary; these might well be artifacts from solving the discrete problem \eqref{equ_expo_fanbeam_example_discrete_equation}.

To better illustrate the values of $\widetilde f^*_{\attenuation}$, we visualize them in Figure \ref{Fig_numerics_values_along_lines} along the central vertical and horizontal lines $L_1$ and $L_2$;
%associated with $\tvariable\mapsto \gamma_{\lambda_1}^\mathrm{fan}(\alpha_1,\tvariable)$ and $\tvariable\mapsto \gamma_{\lambda_2}^\mathrm{fan}(\alpha_2,\tvariable)$
 see Figure \ref{Fig_numerics_least_squares} a). As can be seen, along $L_1$ there are some strong negative and positive values when entering $\Omega$, followed by lighter negative values before reaching the main `hot' activity in the center of the image described above. 
Due to the attenuation effect, even though there appear to be more positive than negative values, the weighted integral is roughly zero.
In contrast, along $L_2$ there are only positive values, used to achieve the corresponding positive value in $g_2$.

Equation \eqref{equ_expo_fanbeam_example_discrete_equation}, was solved to an extremely high accuracy by $\widetilde f^*_{\attenuation}$, with relative error
\begin{equation}\label{equ_appendix_relative_error_expo_small_reso}
\frac{\| \widetilde g - \widetilde \ExpoFanbeam^{\lambda_1,\lambda_2}\widetilde f_{\attenuation}^*\|_{2}}{\|\widetilde g\|_{2}}\approx 10^{-16},
\end{equation}
 suggesting that a solution to \eqref{equ_expo_fanbeam_example_discrete_equation}, and not just a least squares approximation, was found. (Note that $\|\cdot\|_2$ denotes to the classical Euclidean norm on the finite-dimensional space $\RR^{2\times 400}$.) In particular, the projections $\widetilde {\ExpoFanbeam}^{\lambda_1,\lambda_2}\widetilde f^*_\omega$ are virtually identical to $g$ (or rather $\widetilde g$); see Figure \ref{Fig_numerics_least_squares} b). 

 From $\widetilde f^*_{\attenuation}$ satisfying \eqref{equ_expo_fanbeam_example_discrete_equation}, one concludes that it also satisfies \eqref{equ_expo_fanbeam_example_continuous_equation} (when interpreted as a piecewise constant function $f^*_{\attenuation}$) as finding such a function was our original goal. And indeed, $f^*_{\attenuation}$ appears to be approximately projected onto $g$; thus Figure \ref{Fig_numerics_least_squares} a) is somewhat representative of the kind of function we were seeking to illustrate the dense range of $\ExpoFanbeam^{\lambda_1,\lambda_2}$.  
 
%However, theoretically justifying this claim is a lot more nuanced than one might expect and warrants a much more technical discussion. Appendix \ref{Section_appendix_validation} details the challenges and solutions to justifying \eqref{equ_expo_fanbeam_example_continuous_equation}. These considerations are only tangential to the topic of range conditions but might, nonetheless, interest readers as they are representative of a general issue when transfering properties from operators to discretizations and vice-versa. 

\begin{figure}
\center
\newcommand{\mywidth}{0.42}
 \begin{overpic}[height=\mywidth\textwidth]{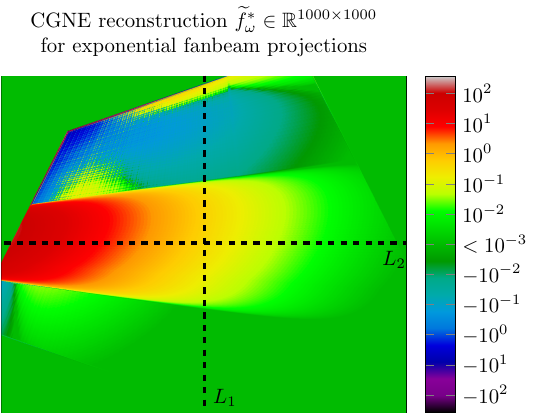}
\put(0,72){\color{black}a)}
\end{overpic}
\begin{overpic}[height=0.42\textwidth]{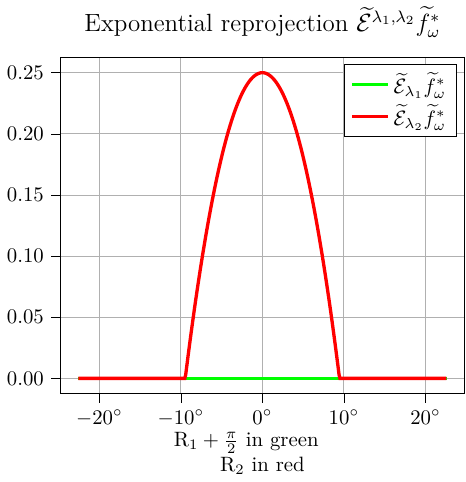}
\put(5,92){\color{black}b)}
\end{overpic}

\caption{Illustration of the least-squares solutions $\widetilde f_{\attenuation}^*\in \RR^{\myresolution\times \myresolution}$ to \eqref{equ_expo_fanbeam_example_discrete_equation} found via 10000 CGNE iterations in a), while b) shows the corresponding projections $\widetilde \ExpoFanbeam^{\lambda_1,\lambda_2} \widetilde f^*_{\attenuation}$. The dashed lines $L_1$ and $L_2$ represent the central horizontal and vertical lines related to the two projections.} 

\label{Fig_numerics_least_squares}
\end{figure}

\begin{figure}

\newcommand{\mywidth}{0.48}

\begin{overpic}[width=\mywidth\textwidth]{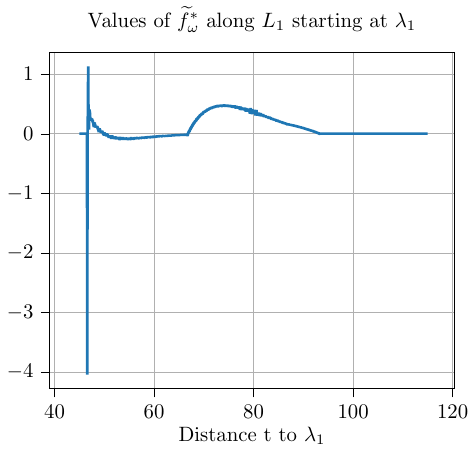}
\put(10,89){\color{black}a)}
\end{overpic}
\hfill
\begin{overpic}[width=\mywidth\textwidth]{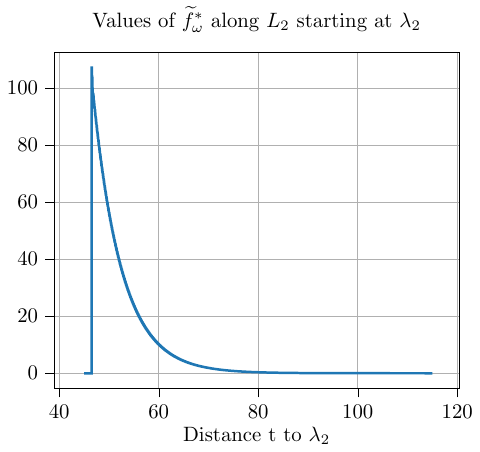}
\put(10,89){\color{black}b)}
\end{overpic}
\caption{Illustration of the values of $\widetilde f^*_{\attenuation}$ along the lines $L_1$ and $L_2$, (see Figure \ref{Fig_numerics_least_squares}), where the parameter $\mathrm t$ represents the distance to the beam divergent point in question.}

\label{Fig_numerics_values_along_lines}

\end{figure}

\section*{Discussion and Conclusion}
\addcontentsline{toc}{section}{Discussion and Conclusion}
This work introduced the concept of projection pair range conditions (Definition \ref{def_LPPRC_kernel}) to find a unified theory for range conditions of tomographic projection pair operators. 
%Naturally, these conditions are also range conditions for projection operators with more than two projections, though they do not necessarily completely characterize the range.
 Investigation of these \conditionname s yielded the kernel conditions (Theorem \ref{thm_kernel_condition}) that need to be satisfied by projection kernels of projection pair range conditions. These offer a systematic approach for identifying \conditionname s and discussing their existence.

%Many known range conditions fall into the presented framework, which has the potential to help in the development of new range conditions. 
%	Given the diverse field of tomography, an effective tool for identifying range conditions of projection pair operators will be highly appreciated.

The uniqueness result of Theorem \ref{thm_kernel_condition} implies that all \conditionname s for a specific projection pair operator are equivalent. Hence, for projection pair operators for which \conditionname s are already known, there are no further or different \conditionname s.

 It will be of interest whether an analogous methodology for finding range conditions for triples or larger tuples can be developed. Though this should be possible in principle, solving the related equations might be very challenging and might be the subject of future work.
 
The fact that no \conditionname s exist for the exponential fanbeam transform  (Theorem \ref{thm_no_LPPRC_for_expo_fanbeam}) is novel. This seems to be the first projection pair operator for which it has been shown that there is no overlap of information between two projections.
As mentioned in Theorem \ref{thm_exp_fanbeam_dense_range}, this means any projection data can be approximated arbitrarily well via the exponential fanbeam transform of smooth functions.
%One such alternative condition is the alternative range condition in Theorem \ref{thm_alternative_range_condition}, illustrating that even though no \conditionname s are possible, there is still a meaningful overlap of information. However, future works will need to show whether this condition is also of practical use.

All of the considerations here are in only two spatial dimensions, but it seems likely that analogous results can be achieved in three dimensions. 

In order to investigate whether projection pair operators that do not admit projection kernels (such as the exponential fanbeam transform) are surjective, future work might investigate whether or not the range of such operators is closed.

\section*{Acknowledgement}
The authors acknowledge primary support from the French `Agence National de la Recherche' via grant ANR-21-CE45-0026 `SPECT-Motion-eDCC’.
Richard Huber was additionally supported by The Villum Foundation (Grant No.25893). {This research was funded in whole or in part by the Austrian Science Fund (FWF) 10.55776/F100800.}

\addcontentsline{toc}{section}{References}
{
\begin{singlespace}
\footnotesize
\bibliographystyle{siamplain}
\bibliography{references}
\end{singlespace}
}

\end{document}